\documentclass[12pt]{amsart}

\usepackage{amsmath}
\usepackage{amsthm}
\usepackage{amssymb}
\usepackage{amscd}
\usepackage{mathrsfs}
\usepackage{latexsym}
\usepackage{graphics}
\usepackage[dvips]{graphicx}
\usepackage{psfrag}
\usepackage{color}
\input{amssym.def}
\usepackage[all]{xy}
\usepackage{hyperref}

\newcommand{\Z}{{\mathbb Z}}
\newcommand{\R}{{\mathbb R}}
\newcommand{\C}{{\mathbb C}}
\newcommand{\M}{{\mathbb M}}
\newcommand{\N}{{\mathbb N}}
\renewcommand{\H}{{\mathbb H}}
\newcommand{\T}{{\mathbb T}}
\newcommand{\Id}{{\mathrm{Id}}}

\newcommand{\Bl}{{\mathcal B}}
\newcommand{\Dl}{{\mathcal D}}
\newcommand{\El}{{\mathcal E}}
\newcommand{\Fl}{{\mathcal F}}
\newcommand{\Hl}{{\mathcal H}}
\newcommand{\Hlb}{{\mathcal Hb}}
\newcommand{\Hlbo}{{\mathcal Hbo}}
\newcommand{\Ol}{{\mathcal O}}
\newcommand{\Pl}{{\mathcal P}}
\newcommand{\Ml}{{\mathcal M}}
\newcommand{\Sl}{{\mathcal S}}
\newcommand{\Tl}{{\mathcal T}}
\newcommand{\Kl}{{\mathcal K}}
\newcommand{\Wl}{{\mathcal W}}

\newcommand{\hc}{\mathfrak{I}}
\newcommand{\hq}{\mathrm{Iq}}
\newcommand{\gop}{\widehat{\Gamma}_0}
\newcommand{\two}{\mathrm{Sym}}

\newcommand{\area}{{\rm area}}
\newcommand{\dist}{{\rm dist}}
\newcommand{\vol}{{\rm vol}}
\newcommand{\covol}{{\rm covol}}
\newcommand{\diag}{{\rm diag}}
\newcommand{\supp}{{\rm supp}}
\newcommand{\grad}{{\rm grad}}

\newcommand{\SLR}{{{\rm SL}_2(\R)}}
\newcommand{\SLZ}{{{\rm SL}_2(\Z)}}
\newcommand{\SLdR}{{{\rm SL}_d(\R)}}

\newcommand{\SLdZ}{{{\rm SL}_d(\Z)}}
\newcommand{\SLdmZ}{{{\rm SL}_{d-1}(\Z)}}

\newcommand{\SLdpmR}{{{\rm SL}_{d-1}^{\pm }(\R)}}
\newcommand{\SO}{{{\rm SO}(2)}}
\newcommand{\SOd}{{{\rm SO}(d)}}
\newcommand{\SOdm}{{{\rm SO}(d-1)}}
\newcommand{\gd}{G_d}
\newcommand{\gdm}{G_{d-1}}
\newcommand{\gad}{\Gamma_d}
\newcommand{\gadm}{\Gamma_{d-1}}

\long\def\comment#1\endcomment{}

\newtheorem{thm}{Theorem}[section]
\newtheorem{cor}[thm]{Corollary}
\newtheorem{lem}[thm]{Lemma}
\newtheorem{prop}[thm]{Proposition}

\theoremstyle{definition}

\newtheorem{example}[thm]{Example}
\newtheorem{remark}[thm]{Remark}
\newtheorem{remarks}[thm]{Remarks}
\newtheorem{notation}[thm]{Notation}
\newtheorem{convention}[thm]{Convention}

\calclayout

\begin{document}

\title[Effective equidistribution of expanding horospheres]{Effective
  equidistribution of expanding horospheres in the locally symmetric space $\mathbf{SO(d) \backslash SL_d(\R) / SL_d(\Z)}$}

\author{C. Dru\c{t}u}
\address{Mathematical Institute, Andrew Wiles Building, Radcliffe Observatory Quarter, Woodstock Road, Oxford OX2 6GG, United Kingdom}
\email{cornelia.drutu@maths.ox.ac.uk}
\urladdr{http://people.maths.ox.ac.uk/drutu/}
\author{N. Peyerimhoff}
\address{Department of Mathematical Sciences, Durham University, Upper Mountjoy Campus, Stockton Road, Durham, DH1 3LE, United Kingdom}
\email{norbert.peyerimhoff@durham.ac.uk}
\urladdr{http://www.maths.dur.ac.uk/~dma0np/}

\date{\today}

\begin{abstract}
We use a dictionary between lattice point counting inside dilated $d$-dimensional ellipsoids (Euclidean counting) and counting of lifts of a closed horosphere that intersect a ball of increasing radius to obtain two types of results. Firstly, \textit{via} an $L^2$-integral error estimate for Euclidean counting, we prove effective equidistribution results for a family of expanding horospheres in the locally symmetric space $SO(d) \backslash SL_d(\R) / SL_d(\Z)$. Secondly, we derive from uniform error estimates in Euclidean counting, error terms for counting $\SLdZ$-orbit points in a certain increasing family of subsets in $SO(d) \backslash SL_d(\R)$ (which we call {\emph{truncated chimneys}}), and for counting the number of lifts of a closed horosphere that intersect a ball with large radius.
\end{abstract}

\maketitle

\section{Introduction}

\subsection{Equidistribution of horospheres}\label{subsec:first} 

The main topic of this paper is the effective equidistribution of expanding periodic horospheres for a particular class of higher rank locally symmetric spaces and geodesic rays inside them. %These horospheres are closed orbits of certain solvable subgroups that are semidirect products of maximal unipotent subgroups with $\mathbb Q$--split tori of dimension $d-2$, and are contained in maximal parabolic subgroups. %Our results permit a prediction of what the best bound on the error term of the equidistribution for maximal unipotents should be. 

There are many equidistribution results for orbits of subgroups of semisimple Lie groups in locally symmetric spaces, in particular the equidistribution of expanding horospheres (and lower dimensional horocyclic sets) is a consequence of Ratner's Theorem. %We refer to \cite{EskinMcM} for an equidistribution result for uniform probability measures on closed orbits of subgroups of the kind that we consider (and of a more general kind).
Providing an estimate of the equidistribution rate (an ``effective'' result) on the other hand is a challenging open problem with important applications \cite{MargulisPbs}. There are few results estimating equidistribution rates, and even fewer attempting to obtain a sharp estimate of the equidistribution rate. Knowledge of such a rate is an important piece of information about the regularity and rigidity of the dynamics involved. Foundational papers in this direction are the ones of Zagier \cite{Zag} and Sarnak \cite{Sarnak}, investigating equidistribution for closed horocycles in finite area hyperbolic surfaces. More precisely, a finite area non-compact hyperbolic surface $M$ has finitely many cusps, and each cusp is covered by a one-parameter family of closed horocycle curves. Each such family $\Fl$ contains a unique closed horocycle curve $\Fl (y)$ of length $\frac{1}{y}$ for $y > 0$, and $\Fl (y)$, endowed with the probability measure $\nu_y$ that puts uniform mass on it, is asymptotically equidistributed with respect to the hyperbolic area measure as $y \to 0$.   

Sarnak \cite{Sarnak} proved the following effective equidistribution result, using properties of Eisenstein series: there exists a finite sequence of numbers $1>s_1>s_2>\dots > s_k >\frac{1}{2}$ and of elements $\mu_1,\dots, \mu_k$ in the dual of $C^1_c (M)$, the space of continuously differentiable functions of compact support on $M$, such that for every $f\in C^1_c (M)$
$$
\oint_{\Fl (y)} f\, d\nu_y = \oint_M f\, d\, {\mathrm{area}}_M + y^{1-s_1}\mu_1 (f)+ \dots + y^{1-s_k}\mu_k (f) +o(y^{1/2}).
$$ 
The $s_i$'s are the poles of the determinant of a matrix valued function which relates Eisenstein series vectors at $s$ and at $1-s$ {\it{via}} a functional equation.

\begin{notation}
The symbol $\oint$ denotes the normalized integral, that is,
  $\oint_A = \frac{1}{\vol (A)} \int_A$. 
\end{notation}

\medskip     

In the particular case of the modular surface ${\mathbb{H}}^2/SL(2, {\mathbb{Z}})$, Zagier \cite{Zag} proved the result above with remainder $o(y^{1/2})$, and showed that a remainder $O(y^{\frac{3}{4}-\varepsilon})$, for every $\varepsilon >0$, would be equivalent to the Riemann Hypothesis. See \cite{Alberto} for a further discussion of Zagier's results, and \cite{CC10,Estala} for similar statements equivalent to the Riemann Hypothesis.

Broadly speaking, the methods used up to now to obtain effective equidistribution results are the following. The method of Zagier \cite{Zag} and Sarnak \cite{Sarnak}, extended to other settings by Marklof \cite{Marklof:theta}, Cacciatori \& Cardella \cite{CC10}, Estala-Arias \cite{Estala} and Romero \& Verjovsky \cite{RV19}, has been to consider Eisenstein series associated to a measure similar to $\nu_y$, and their analytic extensions. 

Another method has been the use of the thickening technique introduced in Margulis' thesis \cite{MargThesis}. This method, usually applied for proving equidistribution results, was made effective in \cite{KleinM96} and inspired  other  effective equidistribution results in  in \cite{Li15,LM18,DKL16,Shi}. 

Finally, the approach used in the work of Burger \cite{Bur90}, Flaminio--Forni \cite{FF03}, Kontorovich-Oh \cite{KOh11, KOh12}, Str\"{o}mbergsson \cite{Strom13,Strom}, S\"{o}dergren \cite{Sode} and more recently Browning-Vinogradov \cite{BV16} and  Edwards \cite{Edwards17,Edwards17b} has been representation-theoretic.
  
\medskip

The current paper uses neither of the above methods, but instead combines existing estimates for the counting of lattice points in Euclidean spaces (in particular a second moment formula due, in its initial form, to Rogers and Schmidt), with a dictionary between the geometry of Euclidean lattices and the geometry of arithmetic lattices in semisimple groups, and of the corresponding locally symmetric spaces. The connection between equidistribution and counting problems in the setting of lattices in semisimple groups is well established (see for instance the influential papers by  Duke-Rudnick-Sarnak \cite{DRS93} and Eskin-McMullen \cite{EskinMcM}). While homogeneous dynamics has proved to be a powerful tool to obtain results in geometric number theory, in this paper the arguments are in the opposite direction: we use Euclidean lattice point results to derive effective horospherical equidistribution with explicit error estimates that improve the currently known ones. We also derive effective counting results in the symmetric space $\SOd \backslash \SLdR$ that are interesting in their own right.

\subsection{Goal of the paper} We investigate from the effective equidistribution viewpoint a special family of expanding horospheres in the higher rank locally symmetric space $\Pl_d/ \SLdZ$. Here $\Pl_d = \SOd \backslash \SLdR$ is  the symmetric
space of non-compact type composed of positive
definite quadratic forms of determinant one on $\R^d$. 

The expanding horospheres that we consider are those associated to the maximal singular geodesic ray $r(t) = \SOd \cdot a_t$ with
\begin{equation}\label{eq:at}
  a_t = \diag\left( e^{\frac{\lambda t}{2}},\dots,e^{\frac{\lambda t}{2}},e^{-\frac{\mu t}{2}} \right), \mbox{ where }\lambda =  \frac{1}{\sqrt{(d-1)d}},\, \mu=
  \sqrt{\frac{d-1}{d}}.
\end{equation}
Since the space $\Pl_d$ has non-positive sectional curvature, the unit speed ray $r: [0,\infty) \to \Pl_d$ gives rise to a
\emph{Busemann function} $f_r: \Pl_d \to \R$ and the corresponding \emph{horospheres} $\Hl_r(t) = f_r^{-1}(t)$ and \emph{open horoballs} $\Hlbo_r(t) = f_r^{-1}((-\infty ,t)), t\in \R ,$ are its level hypersurfaces and respectively strict sublevel sets (see Section \ref{sect:prelim} for precise definitions). 

\begin{notation}\label{notat:gammadpi}
Henceforth, we use the notation $\gd= \SLdR$ and $\gad = \SLdZ$.
We let  $\pi: \Pl_d \to \Ml_d = \Pl_d/\gad$ denote the canonical projection. For any quadratic form $Q \in \Pl_d$, we denote its image $\pi(Q) = Q \cdot \gad$ by $\bar Q$.
\end{notation}

The projections of horospheres $\pi (\Hl_r(t) )$ are immersed hypersurfaces of
finite volume (with isolated orbifold singularities) which can be
viewed as $(d-1)$-torus bundles over copies of the lower
dimensional locally symmetric space $\Ml_{d-1} = \Pl_{d-1} /
\gadm$. These projected horospheres form a family of
hypersurfaces of $\Ml_d$ that expand as $t$ increases to $+\infty$, and our main
result is an effective equidistribution estimate for these hypersurfaces.

\subsection{Equidistribution results}\label{subsec:2ndmainres}  
Our first main result is the following effective horospherical equidistribution for certain $t$-values in shrinking intervals:

\begin{thm} \label{thm:main1} Let $d \ge 2$ and let $\Hl_r(t), t\geq 0,$ be the family of expanding horospheres associated to the geodesic ray defined in \eqref{eq:at}. Let $f \in C(\Ml_d) \cap L^2(\Ml_d)$.
  
For every $\varepsilon >0$ and $\kappa >0$ there exist
  constants $T_0$ and $C_0$ such that the following holds. Every interval
  $$ I_T = [T,T+C_0e^{-\varepsilon T}] $$
  with $T \ge T_0$, contains a parameter $t \in I_T$ such that
  \begin{equation}\label{eq:in1}
    \left\vert \oint_{\pi(\Hl_r(t))} f\, d\vol_{\pi(\Hl_r(t)) }- \oint_{\Ml_d} f\, d\vol_{\Ml_d} \right\vert 
    <  \kappa\,  e^{ - \frac{\sqrt{(d-1)d}}{4} t + \varepsilon t}. 
  \end{equation}
\end{thm}

\begin{remark}
Note that the intervals $I_T \subset \R$ in the theorem are shrinking as $T \to \infty$, and that there is a pay-off between the exponential decay of these intervals and the equidistribution error bound in \eqref{eq:in1}.
\end{remark}

\begin{remark}
The constants $T_0$
  and $C_0$ in the theorem are explicitly given by
  \begin{eqnarray*}
  T_0 &=& \frac{1}{\varepsilon}\, \log \left( \frac{\sqrt{(d-1)d}+4\varepsilon}{2\kappa}\, C_d\, \Vert f \Vert_{\Ml_d}\right) , \\
  C_0 &=& \frac{4}{\kappa}\, C_d\, \Vert f \Vert_{\Ml_d},\mbox{ with }\Vert f \Vert_{\Ml_d}^2 = \oint_{\Ml_d} f^2\, d\vol_{\Ml_d}\mbox{,  and} 
  \end{eqnarray*}
  \begin{equation} \label{eq:Cd} 
  C_d = 2 \sqrt{\frac{2 \zeta(d)}{d(d-1)\, \omega_d}}, 
  \end{equation}
  where $\omega_d$ is the volume of the $d$-dimensional Euclidean unit  and $\zeta$ is the Riemann zeta function.
\end{remark}

% Theorem \ref{thm:main1} seems to suggest, at least for $d \geq 5$, a
% lower bound on the error term that can be obtained for the
% equidistribution of a closed orbit of the horocyclic flow as its
% inner radius goes to infinity. Indeed, if an error term better than
% the one in Theorem \ref{thm:main1} could be obtained for every such
% closed horocyclic orbit, then a calculation done backwards would yield
% a large family of ellipsoids for which the error term in the lattice
% point counting in the corresponding expanding ellipsoids is better
% than the one of G\"otze. It is nevertheless generally believed that
% the exponent in the error term of G\"otze in \cite{G04} cannot be
% improved for generic ellipsoids, even though this is only proved for
% rational ellipsoids.

\medskip

In our second equidistribution
result, formulated below, we use the standard notation $\Vert \cdot \Vert_\infty$ for the maximum norm of functions and of vector fields.

\begin{thm} \label{thm:main2}
Let $d \ge 2$ and let $\Hl_r(t)$ be as in Theorem
  \ref{thm:main1}. There exists a constant $T_d$ such that for all $t \ge T_d$ and $f \in C^1(\Ml_d) \cap L^2(\Ml_d)$, we have  
  \begin{equation} \label{eq:inthoro}
    \left\vert \oint_{\pi(\Hl_r(t))} f\, d\vol_{\Hl_r(t) }- \oint_{\Ml_d} f\, d\vol_{\Ml_d} \right\vert 
    \le \left( C_d \Vert f \Vert_{\Ml_d} + 4 \Vert \grad f \Vert_\infty
    \right) e^{ - \frac{\sqrt{(d-1)d}}{8}  t},
  \end{equation}
  where $C_d$ is the constant given in \eqref{eq:Cd} and $\Vert f \Vert_{\Ml_d}^2 = \oint_{\Ml_d} f^2\, d\vol_{\Ml_d}$.
  Moreover, $T_d$ is explicitly given by
  $$ T_d = \frac{8}{\sqrt{(d-1)d}} \log\left( \frac{3}{4}\sqrt{(d-1)d} \right). $$
\end{thm}

A crucial ingredient in the proof of Theorems \ref{thm:main1} and \ref{thm:main2} is an $L^2$-integral error estimate for primitive lattice points in Borel sets going back to C.~A. Rogers \cite{Rog} and W.~Schmidt \cite{Schm60}.

\medskip

For $d=2$, the exponent in Theorem \ref{thm:main1} is the same as the one obtained by Zagier \cite{Zag} and Sarnak \cite{Sarnak} in the result mentioned in the beginning of Section \ref{subsec:first}, modulo a renormalization of the metric on the hyperbolic plane, as explained in Example \ref{ex:hypplane1}. While the estimate in Theorem \ref{thm:main1again} does not cover all (large enough) values of $t$, as in Zagier and Sarnak's theorem, its interest may arguably come from the fact that its proof is based on a second moment formula  going back to Rogers and Schmidt. See Remark \ref{rem:comparison2} for further details.  

\medskip

A different but related effective equidistribution result in the case $d=3$ has been considered in \cite{LM18}, that is, effective equidistribution \emph{of rational points} on expanding
horospheres.

\medskip

Another related set of results is the one concerning effective equidistribution of orbits of unipotent subgroups (also called orbits of the horocyclic flow). These latter estimates have been the subject of lively research recently, as they are both interesting in their own right, and connected to various applications, e.g. effective error term for the asymptotic distribution of Frobenius numbers \cite{Marklof,Li15} or limit distribution of $n$-point correlation functions of random linear forms \cite[Theorem 3.2]{Marklof2000}.      

We note that Theorem \ref{thm:main2} cannot be obtained from any of the existing estimates for the error term in the equidistribution of
orbits of the horocyclic flow.

Such an estimate was provided implicitly by Li in \cite[Theorem 1]{Li15}. As explained in \cite{Li15}, the use of \cite[Theorems 2.4.3 and A.4]{KleinM96} and the calculation of $p_K(\gd)$ provided in \cite[Section 7]{Oh02} yield an effective version of Theorem 1 from \cite{Li15}.

The best equidistribution rate, to our knowledge, follows from the results of Edwards \cite[Theorem 1]{Edwards17b} combined with Oh's bound on the decay rate of matrix coefficients \cite[Corollary 5.8 and Subsection 6.1]{Oh02} and it is as follows.

Let $U^+ = \{ u \in \gd : \lim_{t \to \infty} a_t u a_{-t} = e \}$ be the expanding horospherical subgroup with respect to $(a_t)_{t\geq 0}$, where $a_t$ is defined in \eqref{eq:at}, $\gad^u = \gad \cap U^+$, and $\mathcal{S}^m(\gd/\gad )$ be the Sobolev space with norm $\Vert \cdot \Vert_{\mathcal{S}^m}$ introduced in \cite[Section 2.3]{Edwards17b}.

\begin{thm}[see {\cite[Theorem 1]{Edwards17b}}]\label{thm:Edwards}
  For every $\varepsilon > 0$ and every $z= \begin{pmatrix} g' & 0 \\ 0 & 1 \end{pmatrix}$ with $g'\in \gdm$, there exists a constant $C_\varepsilon (z) > 0$ such that, for every $f \in \mathcal{S}^m(\gd/\gad )$ and $t \ge 0$:
  \begin{equation} \label{eq:edw}
  \left| \oint_{U^+ / \gad^u} f(a_{-t}z u\cdot \gad ) du
    - \oint_{\gd/\gad} f d\mu \right| \le C_\varepsilon (z) \Vert
  f \Vert_{\mathcal{S}^m}\, e^{- \frac{t}{4}\sqrt{\frac{d}{d-1}} + \varepsilon t},
\end{equation}
where $d\mu$ is the Haar measure of $\gd/\gad$.
\end{thm}

Geometrically speaking, the relation between the convergence in \eqref{eq:edw} and the one in \eqref{eq:inthoro} is as follows: the integration in \eqref{eq:inthoro} is over the domain $\pi(\Hl_r(t))$ and the latter is a torus fiber bundle over a copy of the locally symmetric space $\Ml_{d-1}=SO(d-1)\backslash \gdm/\gadm$, where the fibers are tori that dilate and equidistribute more and more as $t\to \infty $. Theorem \ref{thm:Edwards} provides a rate of equidistribution for each of these dilated torus fibers.

\medskip

While Edwards' Theorem \ref{thm:Edwards} is concerned with effective equidistribution of different closed orbits of the horospherical subgroup $U^+$ which are translates under the 1-parameter family $(a_{-t})$ (in the spirit of Sarnak and Zagier), there is another horospherical equidistribution result for fixed dense horospherical orbits, due to Dani \cite{Dani86}. Recent effective versions of Dani's horospherical equidistribution
theorem were given in \cite[Theorem 3.1]{McAdam18} and in \cite[Theorem 1.11]{Katz19}.

\medskip

In the case of functions with compact support, both Theorems \ref{thm:main1} and \ref{thm:main2} can be further improved by restricting the integration domain: instead of taking the entire base space $\Ml_{d-1}$ of the torus fiber bundle, it suffices to take a compact subset obtained by cutting off the cusp along a sufficiently high horosphere. The height of the horosphere increases with $t$ though, thus in the inequalities \eqref{eq:in1} and \eqref{eq:inthoro} corresponding to the value $t$ one can only replace $\pi(\Hl_r(t))$ by the fiber bundle with base $\Ml_{d-1} \setminus \Hlbo_{\bar \rho }(-\alpha t)$, for some well chosen constant $\alpha >0$ and geodesic ray $\bar \rho$ in $\Ml_{d-1}$. We refer to Theorem \ref{thm:main3again} and Corollary \ref{cor:main3again} for details.

\subsection{Counting results}

The dictionary between counting estimates for \emph{individual} ellipsoids and the geometry of symmetric spaces allows us also to derive counting results for $\gad$-orbits in certain increasing geometrically defined domains which we call \emph{truncated chimneys} (see Theorem \ref{thm:orbitcount}), and for $\gad$-orbits of horospheres intersecting increasing balls. For these results we use lattice point counting estimates by Huxley \cite{Hux} (for $d=2$), Guo \cite{Guo} (for $d = 3,4$) and Bentkus \& G\"otze and G\"otze \cite{BG99,G04} (for $d \ge 5$). For the reader's convenience, we present our horosphere counting result.    

\begin{thm}\label{thm:orbitcountI} (see Theorem \ref{thm:orbitcount2}) 
  Let $\Hl_r = \Hl_r(0)$
  be the level zero horosphere associated to the maximal singular
  geodesic ray from \eqref{eq:at}.
Given an arbitrary point $x$ in the symmetric space $\Pl_d$, we have the following asymptotics for the number of horospheres in the orbit $\gad \Hl_r$ that are at distance at most $T$ from $x$:
  $$ \frac{\# \{ \Hl \in \gad \Hl_r \mid \Hl \cap B(x, T)\neq \emptyset \}}{\kappa_d\, e^{\frac{T}{2}\sqrt{(d-1)d}}/\vol(\Ml_d)} = 1 + \begin{cases} 
  O\left(T^{\frac{18627}{8320}} \exp\left(-\frac{285}{416 \sqrt{2}}T\right)\right) & \text{if $d=2$,}\\
 O\left( \exp\left(-\frac{243}{158\sqrt{6}}T\right) \right) & \text{if $d=3$,} \\
 O\left( \exp\left(-\frac{43\sqrt{3}}{104}T\right) \right) & \text{if $d=4$,} \\
 O\left( \exp\left(-\sqrt{\frac{d-1}{d}}T\right) \right) & \text{if $d\ge5 $.}
  \end{cases}
$$
In the above $\kappa_d$ is an explicit constant depending on $d$ only.
\end{thm}

\medskip

See for comparison Theorem 5 in \cite{DKL16}, where the only effective counting of horospherical objects in higher rank that we are aware of is provided. The result in \cite{DKL16} holds in a more general setting, and concerns slightly different objects (what is called ``horosphere'' in \cite{DKL16} is in fact a horocycle, an orbit of a unipotent subgroup). 
 
\subsection{Plan of the paper} The paper is organised as follows. In Section
\ref{sec:bassymmspace}, we introduce our family of expanding
horospheres in the symmetric space $\Pl_d$ and associated Haar
measures. In Section \ref{sec:volcalclocsym}, we discuss the projection of
these horospheres in the finite volume locally symmetric quotient
$\Ml_d = \Pl_d / \gad$. Following classical arguments of
Siegel, we prove a useful volume result for these projected
horospheres (Proposition \ref{prop:siegel}) involving the Riemann zeta
function. In Section \ref{sec:latcount} we present the relevant
lattice point counting results for expanding ellipsoids from the
literature. Section
\ref{sec:orbitcounting} contains a reformulation of primitive lattice
point counting problems in ellipsoids as counting problems of
$\gad$-orbits in the symmetric space $\Pl_d$. We use this connection, together with results from
Section \ref{sec:latcount}, to establish
Theorem \ref{thm:orbitcount} and Theorem \ref{thm:orbitcount2} (i.e. Theorem \ref{thm:orbitcountI}). Section \ref{sec:effeqdisthorospheres}
is devoted to the proof of our main results. We use an $L^2$-integral error estimate for primitive lattice points in ellipsoids to derive an integrated
version of the effective equidistribution of horospheres, Proposition \ref{prop:equihoroball0}. We then apply specific exponential
decay results for functions on the real line from Appendix
\ref{app:realline}. These results
\begin{itemize}
\item[(a)] imply that there are always very well equidistributed
  horospheres $\Hl_r(t)$ for specific $t$-values within any shrinking interval $[T,T+C_0 e^{-\varepsilon T}]$ (Theorem \ref{thm:main1}
  above);
\item[(b)] allow to complete the proof of an
  equidistribution result for all $t$-levels (Theorem \ref{thm:main2}).
\end{itemize}

\medskip

\subsection*{Acknowledgements}
We are grateful to the anonymous referee who pointed out an improvement of an earlier effective equidistribution result, by means of an $L^2$-integral error estimate. We also thank Shucheng Yu for helpful communications on the literature related to Rogers' second moment formula. We thank Sam Edwards for pointing out reference \cite{Edwards17b} and for his explanations on how to derive from it the 
explicit equidistribution rate of expanding horocyclic orbits presented in Theorem \ref{thm:Edwards}. We thank Jens Marklof for many helpful suggestions and encouraging feedback, Alberto Verjovsky for relevant references, and Tom Ward for various comments concerning the introduction. We thank Jingwei Guo,  Martin Huxley and Wolfgang M\"uller for useful communications on the known estimates in lattice point counting.

\section{Basics about the symmetric space $\Pl_d$}
\label{sec:bassymmspace}

\subsection{Preliminaries on terminology and notation}\label{sect:prelim}

All vectors in this paper are column vectors unless otherwise stated. 

Let $e_1,\dots,e_k$ be the standard basis in $\Z^k$. The $k \times k$ identity matrix is denoted by $\Id_k$.
We denote by $\diag(c_1,\dots,c_k)$ the diagonal matrix having
entries $c_1,c_2,\dots ,c_k$ on the diagonal.

\medskip

Let $X$ be a complete Riemannian manifold of non-positive
curvature. Two geodesic rays in $X$ are called {\it{asymptotic}}
if they are at finite Hausdorff distance from one another. This
defines an equivalence relation $\sim$ on the set $\mathcal{R}$
of geodesic rays in $X$. The \textit{boundary at infinity of }$X$, denoted by
$\partial_\infty X$, is
the quotient $\mathcal{R}/\sim$. Given $\xi\in \partial_\infty X$ and a geodesic ray $r$ in the equivalence class $\xi$, one writes $r(\infty)=\xi$.

Let $r$ be a geodesic ray in $X$. {\it The Busemann function
associated to $r$} is the function
$$
f_r:X\to \R \, ,\; f_r(x)=\lim_{t\to \infty}[\dist (x,r(t))-t]\; .
$$
%The limit above exists because the function $t\to \dist (x,r(t))-t$ is non-increasing and bounded. 
The Busemann functions of two asymptotic rays in $X$ differ by a
constant (see \cite[Cor. II.8.20]{BH}). 

The level hypersurfaces $\Hl_r(a)=\lbrace
x\in X \; ;\;  f_r(x)= a \rbrace$ are called {\it horospheres},
the sublevel sets 
\begin{equation} \label{eq:Hbr(a)}
\Hlb_r(a)=\lbrace x\in X \; ;\;  f_r(x)\leq a
\rbrace
\end{equation}
and $\Hlbo_r(a)=\lbrace x\in X \; ;\;  f_r(x)< a \rbrace$ 
are called {\it{closed horoballs}}, respectively {\it open horoballs}. 

For $a=0$ we use the simplified notation $\Hl_r$ for the horosphere, and $\Hlb_r$ for the
closed horoball, respectively.

Suppose moreover that $X$ is simply connected. Given an arbitrary point $x\in X$ and an arbitrary point at
infinity $\xi \in \partial_\infty X$, there exists a unique
geodesic ray $r$ with $r(0)=x$ and $r(\infty )=\xi$. This allows to define, for $a,b \in \R$, a natural map from $\Hl_r(b)$ to  $\Hl_r(a)$. Indeed, every geodesic ray with origin $x_b$ in $\Hl_r(b)$ and asymptotic to $r$ extends uniquely to a bi-infinite geodesic in $X$, and the latter intersects $\Hl_r(a)$ in a unique point $x_a$. The map $\Phi_{b,a}$ defined by $x_b\mapsto x_a$ coincides with the nearest point projection of $\Hl_r(b)$ to $\Hl_r(a)$, and it is a diffeomorphism from $\Hl_r(b)$ to $\Hl_r(a)$, with inverse $\Phi_{a,b}$. If $b>a$ then $\Phi_{b,a}$ contracts distances and volumes. 

\begin{notation}
In what follows we use, for simplicity, the notation $\Phi_t$ for the map
$\Phi_{0,t}: \Hl_r(0) \to \Hl_r(t).$ 
\end{notation}

\comment
Therefore we shall sometimes call them
{\it{Busemann functions of basepoint }}$\xi$, where $\xi $ is the
common point at infinity of the two rays. The families of
horoballs and horospheres are the same for the two rays. We shall
say that they are horoballs and horospheres {\it of basepoint
}$\xi$.
\endcomment

\subsection{Geometry of $\Pl_d$}\label{sec:geompd} We briefly recall two
realisations of our symmetric space. We use a notation similar to the
one in \cite[Section 3]{Dr05}.  Let $\Pl_d$ be the space of positive definite quadratic forms on $\R^d$ of determinant one. For
the quadratic form $Q \in \Pl_d$, we denote its evaluation of the
vector $v \in \R^d$ by $Q(v)$, that is,
$$ Q(v) = v^\top M_Q v, $$
where $M_Q$ is the symmetric matrix representing the quadratic form $Q$ in the standard basis. 

The group $\gd$ acts transitively from the right on $\Pl_d$. More precisely, for $g \in \gd$ and $Q \in \Pl_d$, let
$$ (Q \cdot g)(v)= Q(g v) = (g v)^\top M_Q (g v) \quad \text{for all $v \in \R^d$}. $$

We have that
\begin{equation}\label{eq:fac2}
\{ g \in \gd \mid  \text{$Q \cdot g =
Q$ for all $Q \in\Pl_d$} \} = Z(G_d)=\begin{cases} \{ \Id_d \} 
& \text{if $d$ is odd,} \\ \{ \pm \Id_d \} & \text{if $d$ is even.} \end{cases}
\end{equation}
Thus, it is the projectivization ${\rm{P}}G_d$ that acts
faithfully on $\Pl_d$. 

\medskip

The space $\Pl_d$ can be endowed with a distance function $\dist_{\Pl_d}$ that is invariant with respect to the action of $\gd$ \cite[Section 3]{Dr05}.

Let $Q_0$ be the quadratic form corresponding to the standard Euclidean inner product, that is $Q_0(v) = v^\top v$. There is a natural identification of $\Pl_d$ with the homogeneous space $\SOd \backslash \gd$, given by
$$\hq: \SOd \backslash \gd \to \Pl_d, \quad \SOd \cdot g \mapsto
Q_0 \cdot g. $$

The identification $\hq$ is equivariant with respect to the two actions to the right of  $\gd$.

We will often switch between these two representations of our symmetric space, and the corresponding actions.

Next, we introduce the horospheres associated to two (unit speed) geodesic rays
\begin{equation} \label{eq:rt}
r(t) = \SOd \cdot a_t \in \SOd \backslash \gd 
\end{equation}
with
\begin{equation} \label{eq:atagain}
a_t = \diag(e^{\lambda t/2},\dots,e^{\lambda t/2},e^{-\mu t/2}) 
\end{equation}
and
$$
\lambda = \sqrt{\frac{1}{(d-1)d}}\, , \quad 
\mu = (d-1) \lambda = \sqrt{\frac{d-1}{d}}, 
$$
and
\begin{equation} \label{eq:rhot}
\rho (t) = \SOd \cdot \tilde{a}_t \in \SOd \backslash \gd 
\end{equation}
with
\begin{equation} \label{eq:atagainrho}
\tilde{a}_t = \diag(e^{\mu t/2},e^{-\lambda t/2},\dots ,e^{-\lambda t/2}). 
\end{equation}

Note that $r(t)$, viewed as quadratic form, has the symmetric matrix $\diag(e^{\lambda t},\dots,e^{\lambda t},e^{-\mu t}) = (a_t)^\top a_t$ as its representative. Likewise, $\rho (t)$ is the quadratic form with symmetric matrix $\diag(e^{\mu t},e^{-\lambda t},\dots,e^{-\lambda t}) = (\tilde a_t)^\top \tilde a_t$.

Both rays are one-dimensional faces of the \textit{Weyl chamber} $\Wl_0$ composed of quadratic forms with symmetric matrices $\diag(e^{s_1},\dots,e^{s_d}),$ where $ \sum s_i=0 $ and $s_1\geq s_2\geq \dots \geq s_d $.

Let us now discuss the Busemann functions corresponding to the geodesic rays $r$ and $\rho$. Their explicit formulas are given in \cite[Lemma 3.2.1]{Dr05}. The Busemann function $f_r: \Pl_d \to \R$ corresponding to the
geodesic ray $r$ is given by  
\begin{equation} \label{eq:busem} 
f_r(Q) = \lim_{t \to \infty} (\dist_{\Pl_d}(r(t),Q) - t) =  \sqrt{\frac{d}{d-1}} \log Q(e_d).
\end{equation}
Likewise, the Busemann function $f_\rho : \Pl_d \to \R$ corresponding to $\rho$ is explicitly described as follows. 
For every quadratic form $Q$, given $Q^{\perp e_1}$ the restriction of $Q$ to the hyperplane $\langle e_2,\dots , e_d\rangle$, and $\det Q^{\perp e_1}$ the determinant of the matrix representing $Q^{\perp e_1}$ in the basis $\{ e_2,\dots , e_d\}$,    
\begin{equation} \label{eq:busemrho} 
f_\rho (Q) = \sqrt{\frac{d}{d-1}} \log \det Q^{\perp e_1}. 
\end{equation}

The (closed) horoballs of the ray $r$ are therefore
\begin{equation}\label{eq:horo}
 \Hlb_r(t) = {f_r}^{-1}((-\infty,t]) = \left\{ Q \in \Pl_d \mid Q(e_d) \leq e^{t \sqrt{\frac{d-1}{d}}} \right\}, 
\end{equation}
and the horospheres $\Hl_r(t) = {f_r}^{-1}(t)$ are defined similarly by replacing ``$\leq$'' in the right hand side of \eqref{eq:horo} by ``$=$''. 

\begin{notation}
For every $1\leq k <d$ and subsets $A\subset M_{k\times k} (\R ), B\subset M_{k\times (d-k)} (\R ), C\subset M_{(d-k)\times k} (\R ), D\subset M_{(d-k)\times (d-k)} (\R )$, we use the notation
$$
\begin{pmatrix} A & B \\ C & D \end{pmatrix} = \left\{ \begin{pmatrix} a & b 
\\ c & d \end{pmatrix} \in \gd \Bigm| a\in A, b\in B, c\in C, d\in D \right\}. 
$$

We write $\R^k$ instead of $M_{1\times k} (\R )$, or $M_{k\times 1} (\R )$, with the convention that the vectors of $\R^k$ are written as columns or as rows, as appropriate.  
\end{notation}

\medskip

The map $\Phi_{s,t}$ defined as in Section \ref{sect:prelim}, between two horospheres $\Hl_r(s), \Hl_r(t)$ of $\Pl_d$, with $s<t$, is a diffeomorphism that expands the volume exponentially in $t-s$ (see Lemma \ref{lem:horovolrel}, \eqref{eq:volhormod}). Therefore, we refer to the family $\Hl_r(t)$ with $t$ increasing from $0$ to $+\infty$ as the \emph{family of expanding horospheres associated to the ray~$r$}.

\begin{example} \label{ex:hypplane1}
In the case $d=2$, the model $\Pl_2$ of the $2$-dimensional symmetric space   agrees, up to rescaling, with the hyperbolic plane model $\H = \{ z \in \C \mid {\rm{Im}}(z) > 0 \}$ with distance function $\dist_{\H}$ given by
  $$ \sinh\left( \frac{1}{2}\dist_{\H}(z_1,z_2) \right) = \frac{|z_1 -z_2|}{2 \sqrt{{\rm{Im}}(z_1){\rm{Im}}(z_2)}}. $$
  The isometry between the two spaces is given by
  \begin{equation} \label{eq:homhypplane} 
  (\SO \backslash \SLR, \dist_{\Pl_d}) \to ( \H, \sqrt{2}\, \dist_{\H} ), \quad \SO \cdot g \mapsto g^\top \cdot i, 
  \end{equation}
  where $\SLR$ acts on $\H$ in the usual way by M\"obius transformations. Moreover, under this identification we have $r(t) = \rho(t) =  e^{\frac{t}{\sqrt{2}}}$, $f_r(z) = -\sqrt{2} \log {\rm{Im}}(z)$, and $\Hlb_r(t) = \{z \in \H \mid {\rm{Im}}(z) \ge e^{\frac{t}{\sqrt{2}}} \}$.
\end{example}

\medskip

\subsection{Haar measures}\label{sec:Haar} In what follows, we provide an alternative
description of the symmetric space $\Pl_d$ and of the horospheres determined by the ray defined in \eqref{eq:rt} and \eqref{eq:atagain}, in terms of the Iwasawa
decomposition of $\gd$. This third description is, in a certain sense, a system of coordinates, it allows to describe Haar measures on the whole groups and several subgroups, but it is not endowed with an action by isometries of the whole group.  

Recall that the Iwasawa decomposition of $\gd$ is given by
$$ \gd = K_d N_d A_d = K_d A_d N_d,$$ 
with $K_d = \SOd$, $A_d = \left\{ \diag(e^{t_1},\dots,e^{t_d}) \in \gd \mid \sum_{j=1}^d t_j = 0 \right\}$
, and
\begin{equation} \label{eq:Nd} 
N_d = \left\{ n(x) \in \gd\, \Biggm|\, n(x) = \begin{pmatrix} 1 & & 0 \\ & \ddots & \\ x_{ij} & & 1 \end{pmatrix} \, \text{with} \, x_{ij} \in \R \right\}. 
\end{equation}

We can identify our symmetric space $(\Pl_d,\dist_{\Pl_d})$ with the solvable group $S_d = N_d A_d= A_d N_d$ with a right-invariant metric {\it{via}} the identifications
\begin{equation} \label{eq:Iwd} 
S_d \to \SOd \backslash \gd \to \Pl_d, \quad g \mapsto \SOd \cdot g \mapsto Q_0 \cdot g. 
\end{equation}
The identification between $S_d$ and $\Pl_d$ described above will henceforth
be called the \emph{(canonical) Iwasawa map}, and denoted by
$$ \hc_d: S_d \to \Pl_d, $$ 
and we will call $S_d$ the \emph{(canonical) Iwasawa space of coordinates} of the
symmetric space $\Pl_d$. 

The horosphere $\Hl_r(t) \subset \Pl_d$, with $t\in \R$,
is identified {\it{via}} $\hc_d$ with
$a_{-t}N_d A' = a_{-t}A' N_d \subset S_d$, where $A' =  \left\{ \diag(e^{t_1},\dots,e^{t_{d-1}},1) \in \gd \mid
  \sum_{j=1}^{d-1} t_j = 0 \right\}.$
 
The maps $\Phi_t$ defined in the end of Section \ref{sec:geompd} can be described
{\it{via}}  $\hc_d$ as follows:
\begin{equation} \label{eq:Phit}
\Phi_t(\hc_d(g)) = \hc_d(a_{-t} g).
\end{equation}
This new description makes it clear that the maps $\Phi_t$ have natural extensions as diffeomorphisms of the whole space $\Pl_d$.

Next, we describe Haar measures on the groups $N_d$, $A_d$
and $A'$. The identification
$$ N_d \ni n(x) \mapsto x = (x_{ij})_{1 \le j < i \le d} \in \R^{d(d-1)/2} $$
induces a Haar measure $dn$ on $N_d$ as the pullback of a suitable rescaling $c_d \cdot dx$ of the Lebesgue measure 
$$ dx = \prod_{1 \le j < i \le d} dx_{ij} $$
on $\R^{d(d-1)/2}$. The precise value of the rescaling constant $c_d$ is not relevant for us, but it can be checked that we
have $c_d = \sqrt{2}$ when $d=2$.

The map
$$ A_d \ni \diag(e^{t_1},\dots,e^{t_d}) \mapsto 2(t_1,\dots,t_d) \in \R^d$$
is an isometry between the totally geodesic flat $\hc_d(A_d)$ in the
symmetric space $(\Pl_d,\dist_{\Pl_d})$ and the hyperplane 
$$ \Sigma_d = \left\{ (\tau_1,\dots,\tau_d)\in \R^d \Biggm| \sum_{j=1}^d \tau_j = 0 \right\} \subset \R^d $$
with the induced standard Euclidean
metric. We induce a Haar measure $da$ on $A_d$ as the pullback of the associated standard Lebesgue measure on $\Sigma_d$ under this map.  Similarly, we introduce the Haar measure $da'$
on $A'$ as the pullback of the standard Lebesgue measure
of the Euclidean hyperplane $\Sigma_{d-1} \subset \R^{d-1}$ {\it{via}} the identification
$$ A' \ni \diag(e^{t_1},\dots,e^{t_{d-1}},1) \mapsto 2(t_1,\dots,t_{d-1}) \in \Sigma_{d-1}. $$
Then the pullback of the Haar measure
$da$ on $A_d$ under the identification
$$ \R \times A' \cong A_d, \quad (t, a') \mapsto a_{-t} a' $$
coincides with the product measure $dt \, da'$.

We denote the normalized bi-invariant Haar measure on $K_d = SO(d)$ by $dk$. These Haar measures define a
bi-invariant Haar measure $dg$ on $\gd$ {\it{via}}
\begin{equation} \label{eq:intiwa} 
\int_{\gd} f(g) dg = \int_{K_d} \int_{N_d} \int_{A_d}
f(kna) da\, dn\, dk, 
\end{equation}
for any function $f \in C_c(\gd)$.
We check that for every $a =\diag(a_1,\dots,a_d) \in A_d$ 
$$n(x) a = a n(x'), \mbox{ with }x_{ij}' = \frac{a_j}{a_i}  x_{ij}\mbox{ for }1 \le j < i \le d.
$$ 
We can then write
\begin{multline*} 
\int_{N_d} \int_{A_d} f(kna)da\, dn
= \int_{\R^{d(d-1)/2}} \int_{A_d} f(kn(x)a) da\, dx
= \int_{A_d} \int_{\R^{d(d-1)/2}} f(kan(x')) \chi_d(a) dx'\, da \\
= \int_{N_d} \int_{A_d} f(kan) \chi_d(a) da\, dn,
\end{multline*}
where
$$ \chi_d(a) = \prod_{1 \le j < i \le d} \frac{a_i}{a_j}. $$
Therefore,
$$ \int_{\gd} f(g)dg = \int_{K_d} \int_{N_d} \int_{A_d} f(kan) \chi_d(a) da\, dn\, dk. $$
Similarly, we have for $f \in C_c(\Pl_d)$,
$$ \int_{\Pl_d} f d\vol_{\Pl_d} = \int_{N_d} \int_{A_d} f\circ \hc_d(an) \chi_d(a) da\, dn, $$
and for $f \in C_c(\Hl_r(t))$,
\begin{equation} \label{eq:inthor} 
\int_{\Hl_r(t)} f d\vol_{\Hl_r(t)} = \int_{N_d} \int_{A'} f\circ \hc_d(a_{-t}a'n) \chi_d(a_{-t}a') da'\, dn. 
\end{equation}

\begin{lem} \label{lem:horovolrel}
  For any function $f \in C_c(\Hl_r(t))$, we have
  \begin{equation} \label{eq:inthormod}
  \int_{\Hl_r(t)} f d\vol_{\Hl_r(t)} = e^{\frac{t}{2} \sqrt{(d-1)d}} \int_{\Hl_r(0)}
  f \circ \Phi_t\, d\vol_{\Hl_r(0)}. 
  \end{equation} 
  This implies, in particular,
  \begin{equation} \label{eq:volhormod} 
  \vol_{\Hl_r(t)}(\Phi_t(K)) = e^{\frac{t}{2} \sqrt{(d-1)d}} \vol_{\Hl_r(0)}(K). 
  \end{equation}
  Moreover, for any function $f \in C_c(\Pl_d)$, we have
  \begin{equation} \label{eq:fubinihor} 
  \int_{\Pl_d} f d\vol_{\Pl_d} = \int_\R e^{\frac{t}{2} \sqrt{(d-1)d}} \int_{\Hl_r(0)}
  f \circ \Phi_t d\vol_{\Hl_r(0)}\, dt. 
  \end{equation}
All the above remains true if the geodesic ray $r$ is replaced by $\rho$.
\end{lem}

\begin{proof} The identity \eqref{eq:inthormod} follows directly
from \eqref{eq:inthor}, $\chi_d(a 'a) = \chi_d(a) \chi_d(a')$, 
$f\circ \hc_d(a_{-t}g) = f \circ \Phi_t(\hc_d(g))$, and
$$ \chi_d(a_{-t}) = \prod_{j=1}^{d-1} \frac{e^{\mu t/2}}{e^{-\lambda t/2}} = e^{\frac{t}{2}(d-1)(\mu+\lambda)} = e^{\frac{t}{2} \sqrt{(d-1)d}}. $$
Since the horospheres $\Hl_r(t)$, $t \in \R$, form a
foliation of the space $\Pl_d$ by equidistant hypersurfaces, we 
have by Fubini that for all $f \in C_c(\Pl_d)$,
$$ \int_{\Pl_d} f d\vol_{\Pl_d} = \int_\R \int_{\Hl_r(t)} f d\vol_{\Hl_r(t)}\, dt. 
$$
The identity \eqref{eq:fubinihor} follows then directly from
\eqref{eq:inthormod}.

The same proofs carry over verbatim to the case of the geodesic ray $\rho$, except for the following slightly different calculation:
$$
\chi_d(\tilde{a}_{-t}) = \prod_{j=1}^{d-1} \frac{e^{\lambda t/2}}{e^{-\mu t/2}} = e^{\frac{t}{2}(d-1)(\mu+\lambda)} = e^{\frac{t}{2} \sqrt{(d-1)d}}. 
$$
\end{proof}

\section{Volume calculations for locally symmetric quotients}
\label{sec:volcalclocsym}

\subsection{The locally symmetric space $\Ml_d$} \quad

\begin{notation}
Henceforth, we further simplify the notation introduced in \ref{notat:gammadpi} for $\SLdR$ and $\SLdZ$, and we denote them simply by $G$ and $\Gamma$, respectively.

We denote the projectivizations of $G$ and of $\Gamma $ by $PG$ and $P\Gamma$ respectively. If $d$ is odd then both projectivizations coincide with the groups, while if $d$ is even then the projectivizations are the quotients by $Z(G)=\{ \pm \Id_d \}$.
\end{notation}

The difference, in terms of faithfulness of the action, between the odd- and the
even-dimensional case, as pointed out in \eqref{eq:fac2}, needs to be taken into account in our counting arguments presented later on. For this reason, we introduce the following function.

\begin{notation}\label{notat:alphad} 
  Let
  \begin{equation} \label{eq:alpha} 
  \alpha(d) = |Z(G)| = \begin{cases} 1 & \text{if $d$ is odd,} \\
 2 & \text{if $d$ is even.} \end{cases} 
 \end{equation}
\end{notation}

\medskip

The quotient $\Ml_d = \Pl_d/\Gamma  = \Pl_d/P\Gamma$ inherits a metric and a volume
element $d\vol_{\Ml_d}$ from the metric and the volume form
$d\vol_{\Pl_d}$. This
makes $\Ml_d$ a non-compact locally symmetric space of finite volume.

The restriction of the projection $\pi$ (see \ref{notat:gammadpi}) to the Weyl chamber $\Wl_0$ defined in section \ref{sec:geompd} is an isometric embedding, and its image ${\bar{\Wl}}_0$ is at finite Hausdorff distance $c_d$ from $\Ml_d$ \cite[$\S 2$]{Le}.

In particular $\bar r = \pi \circ r$ and $\bar \rho = \pi\circ \rho$ are geodesic rays in $\Ml_d$, therefore according to Section \ref{sect:prelim} they define Busemann functions, horospheres and horoballs. For $a<0$ with $|a|$ large enough, $\Hlb_{\bar r}(a)$ is the projection of $\Hlb_{r}(a)$, and $\Hlb_{\bar \rho}(a)$   is the projection of $\Hlb_{\rho }(a)$. Moreover, this and \eqref{eq:horo} imply that 
\begin{equation}\label{eq:horogamma}
\pi^{-1} \left( \Hlb_{\bar{r}}(a) \right)= \bigcup_{\gamma\in \Gamma} \Hlb_{r}(a)\gamma = \bigcup_{v\in {\widehat \Z}^d}  \left\{ Q \in \Pl_d \mid Q(v) \leq e^{a \sqrt{\frac{d-1}{d}}} \right\}. 
\end{equation}

For all the statements above we refer to \cite[$\S 3.6$]{Dr05} and references therein.

\begin{notation}\label{notat:fld}
Let $\Fl_d \subset \Pl_d$ be a fundamental domain of the
$\Gamma$-right action (and of the $P\Gamma$-right action) on $\Pl_d$, containing   the Weyl chamber $\Wl_0$. 
\end{notation}

%%%%%%%%%%%%%%%%%
%%%%%%%%%%%%%%%%%

\medskip

We now recall some facts about stabilizers and quotients of horospheres. 
\medskip

\begin{remarks}\label{rem:stabhoro}

\begin{enumerate}
\item\label{item:semidir} We begin by noting that, as the actions by isometries in this paper are to the right, we use an adapted version of semidirect product. Given two groups $N,H$ and an action of $H$ on $N$ \emph{to the right}, $(n, h)\mapsto n\cdot h$, the corresponding semidirect product $G=H\ltimes N$ is defined by $(h_1, n_1)(h_2, n_2) = (h_1h_2, (n_1 \cdot h_2) n_2).$ With this definition, we have that $(H \ltimes K )\backslash (H\ltimes G)$ has a $(H\ltimes G)$-equivariant identification to $K \backslash G$, where the right action of $H\ltimes G$ on $G$ is $g\cdot (h, g')= (g\cdot h) g'$.

\medskip
  
\item\label{item:stabr}  It can be derived from \eqref{eq:busem} that, for any $t\in \R$,
\begin{equation*}
{\rm{Stab}}_{\gd}(\Hl_r(t))= {\rm{Stab}}_{\gd}(\Hl_r(0)) = \{ g \in \gd \mid g e_d = \pm e_d \} 
=  \begin{pmatrix} \SLdpmR & 0 \\ \R^{d-1} & \pm 1 \end{pmatrix} ,
\end{equation*}
where $\SLdpmR$ equals $\gdm$ if the lower right corner is $1$, and it equals 
${\rm SL}_{d-1}^-(\R) = \{ g\in {\rm GL}_{d-1}(\R) \mid \det g = -1\}$  if the lower right corner is $-1$. 

With the notation 
\begin{equation} \label{eq:two} 
  \two  = \begin{pmatrix} \Id_{d-2} & 0 \\ 0 & \pm \Id_2 \end{pmatrix} ,
 \end{equation}  
an alternative way of writing $ {\rm{Stab}}_{\gd}(\Hl_r(0))$ is as a semidirect product, as follows
\begin{equation}\label{eq:stabsemidir}
\two\ltimes  \begin{pmatrix} \gdm & 0 \\ \R^{d-1} & 1 \end{pmatrix}.
\end{equation}
When $d$ is even, ${\rm{Stab}}_{\gd}(\Hl_r(0))$ can also be written as a direct product
\begin{equation}\label{eq:stabsemidir2}
Z(G_d) \times \begin{pmatrix} \gdm & 0 \\ \R^{d-1} & 1 \end{pmatrix}.
\end{equation}  

\medskip

\item The description of $\Hl_r(0)$ in the model $\SOd \backslash \gd $, with the identification described in the end of \eqref{item:semidir}, becomes
$$
\begin{pmatrix} \SOdm & 0 \\ 0 & 1 \end{pmatrix}\backslash \begin{pmatrix} \gdm & 0 \\ \R^{d-1} & 1 \end{pmatrix}. 
$$  
\end{enumerate} 
\end{remarks}

\begin{notation}\label{notatG0} Let $\Gamma_0 = \Gamma \cap {\rm{Stab}}_G(\Hl_r(0)) = \Gamma \cap {\rm{Stab}}_G(\Hl_r(t))$ for every $t\in \R$. We denote by $P\Gamma_0$ the projectivization of $\Gamma_0$, equal to $\Gamma_0$ when $d$ is odd. 
Remark \ref{rem:stabhoro}, \eqref{item:stabr}, implies that 
$$
\Gamma_0 = \two \ltimes \begin{pmatrix} \SLdmZ & 0 \\ \Z^{d-1} & 1 \end{pmatrix}
= Z(G) \times \begin{pmatrix} \SLdmZ & 0 \\ \Z^{d-1} & 1 \end{pmatrix}\mbox{ when $d$ is even}.
$$
We denote by $\gop$ the normal subgroup $\begin{pmatrix} \SLdmZ & 0 \\ \Z^{d-1} & 1 \end{pmatrix}$ of $\Gamma_0$. 
\end{notation}

The projection $\Gamma \to P\Gamma$ restricted to $\gop$ is injective, moreover when $d$ is even it is an isomorphism from $\gop$ to $P\Gamma_0$. Thus, we can always see $\gop$ also as a subgroup of $P\Gamma$.  

\medskip 

A projected horosphere $\pi(\Hl_r(t))$ with $t\in \R$, is an immersed hypersurface in $\Ml_d$ (with isolated orbifold singularities) {\it{via}} the parametrization 
$$\Hl_r(t)/\Gamma_0 \to \pi(\Hl_r(t)),\; Q \cdot \Gamma_0 \mapsto \bar Q =
\pi(Q).$$

As mentioned previously, for $t<0$ with $|t|$ large enough, the above parametrization is an isometric embedding.
According to Remark \ref{rem:stabhoro}, $\Hl_r(0)/\Gamma_0$ can be identified with the double quotient
\begin{equation}
\Bl_0=\begin{pmatrix} \SOdm & 0 \\ 0 & 1 \end{pmatrix}\backslash \begin{pmatrix} \gdm & 0 \\ \R^{d-1} & 1 \end{pmatrix} / \gop
\end{equation}
when $d$ is even, and to $\Bl_0/\two $ when $d$ is odd. 

\begin{notation}\label{notatf0}
We denote by $\Fl_0(0)$ a fundamental domain for the action of $\gop $ on the horosphere $\Hl_r(0)$.  
\end{notation}

As explained in Section \ref{sec:Haar}, {\it{via}} the Iwasawa map $\hc_d$ the horosphere $\Hl_r(0)$
can be identified with the subgroup $N_d A' = A' N_d$ of the solvable group $S_d = N_d A_d$.

A fundamental domain of the $\gop$-action on $\Hl_r(0)$ can then be described in the Iwasawa coordinates as follows
\begin{equation} \label{eq:F00} 
\Fl_0(0) = \left\{ \begin{pmatrix} g' & 0 \\ z^\top & 1 \end{pmatrix} \mid g' \in \hc_{d-1}^{-1}(\Fl_{d-1}), z \in [-1/2,1/2)^{d-1} \right\}, 
\end{equation}
where $\Fl_{d-1} \subset \Pl_{d-1}$ is a fundamental domain of the $\SLdmZ$-right action on the symmetric space $\Pl_{d-1}$. The quotient $\Hl_r(0)/\gop $ can be seen as a fiber bundle over the locally symmetric space $\Ml_{d-1}$ with fiber $\R^{d-1} / \Z^{d-1}$, that is, a $(d-1)$-dimensional torus $\T^{d-1}$.  

Likewise, for every $t\in \R$, the quotient $\Hl_r(t)/\gop $ can be identified with a fundamental domain $\Fl_0(t)$ described in the Iwasawa coordinates by $a_{-t} \Fl_0(0)$.

The union 
\begin{equation} \label{eq:S}
\Sl = \bigcup_{t \in \R} \Fl_0(t)
\end{equation}
provides a description in the Iwasawa coordinates of a fundamental domain of the $\gop$-right action on $\Pl_d$.

Note that, when $d$ is even, in all the above, $\gop \simeq P\Gamma_0$ can be replaced by $\Gamma_0$. 

\medskip

\subsection{Siegel's volume formula}
The final aim of this section is to derive the following useful volume
formula. Its proof is based on classical arguments by C. L. Siegel in
\cite{Sie}. We present this proof and follow the expositions given in
\cite{Ga} and \cite{TP} (see also \cite[Section 4.4.4]{Terras}).

\begin{prop} \label{prop:siegel}
  We have
  $$ \alpha(d) \, \frac{\vol(\Hl_r(0)/\gop )}{\vol(\Ml_d)} = \alpha(d) \frac{\vol(\Fl_0(0))}{\vol(\Ml_d)} = \frac{\sqrt{d(d-1)}}{2} \frac{\omega_d}{\zeta(d)}, $$
  where $\omega_d$ is the volume of the $d$-dimensional Euclidean unit ball, 
$\alpha(d)$ is defined in \eqref{eq:alpha}, and $\zeta$ is the Riemann zeta function.
\end{prop}

Before we start the proof, let us introduce the following notation.

\begin{notation} The set of all \emph{primitive lattice points} in $\Z^d$ will henceforth be denoted by ${\widehat \Z}^d$, i.e.,
$$ {\widehat \Z}^d = \{ x = (x_1,\dots,x_d) \in \Z^d\backslash \{ 0 \} \mid \gcd(x_1,\dots,x_d) = 1 \}. $$
\end{notation}

The set ${\widehat \Z}^d$ can be identified with the quotient $\Gamma / \gop $ {\it{via}} the bijective map
$$ \Gamma / \gop \to {\widehat \Z}^d, \quad \gamma \gop \mapsto \gamma e_d. $$

\begin{proof}[Proof of Proposition \ref{prop:siegel}]
  Let $f \in C^\infty(\R^d)$ be a radial Schwartz function, that is,
  \begin{equation}\label{eq:radial}
 f(k x) = f(x)\mbox{ for all }x \in \R^d\mbox{ and }k \in K_d = \SOd \, . 
  \end{equation}
 We define
  a function $S_f: G \to \R$ {\it{via}}
  $$ S_f(g) = \sum_{x \in \Z^d} f(gx) = f(0) + \sum_{x \in \Z^d \backslash \{ 0 \}} f(gx). $$
  Note that $S_f$ is left $\SOd$-invariant and right $\Gamma$-invariant, so it can be viewed as a function on $\Ml_d = \Pl_d / \Gamma$
  and the integral $\int_{\Ml_d} S_f\, d\vol_{\Ml_d}$ is well-defined. 
  Our first aim is to rewrite this integral. We have that
  $$ \Z^d \backslash \{ 0 \} = \bigcup_{\ell \in \N} \ell\, {\widehat \Z }^d
  = \bigcup_{\ell \in \N} \bigcup_{\gamma \in \Gamma / \gop }
   \{ \ell \gamma e_d \}, $$
   and, therefore,
$$
 \int_{\Ml_d} S_f\, d\vol_{\Ml_d} = \int_{\Ml_d} f(0) d\vol_{\Ml_d} +\sum_{\ell \in \N} \sum_{\gamma \in \Gamma/\gop } \int_{\Pl_d/P\Gamma } f(\ell g \gamma e_d) d\vol_{\Ml_d}.
$$

When $d$ is odd, we have $P\Gamma =\Gamma$, while when $d$ is even $P\Gamma =\{ \pm \Id_d\} \backslash \Gamma$ and \eqref{eq:radial} implies that $f(-x) = f(x)$. It follows that 
\begin{eqnarray}
\int_{\Ml_d} S_f\, d\vol_{\Ml_d} &=& \int_{\Ml_d} f(0) d\vol_{\Ml_d} + \alpha(d)\sum_{\ell \in \N} \sum_{\gamma \in P\Gamma/\gop } \int_{\Pl_d/P\Gamma } f(\ell g \gamma e_d) d\vol_{\Ml_d} \nonumber \\
&=& f(0) \vol(\Ml_d) + \alpha(d)\sum_{\ell \in \N} \sum_{\gamma \in P\Gamma/\gop } \int_{\Fl_d } f(\ell g \gamma e_d) d\vol_{\Pl_d} \nonumber
\\
   &=& f(0) \vol(\Ml_d) + \alpha(d) \sum_{\ell \in \N} \int_{\Pl_d/\gop } 
   \Psi_\ell (g) \, d\vol_{\Pl_d/\gop } \label{eq:intsum}
   \end{eqnarray}
   with left $\SOd$-invariant and right $\gop$-invariant
   functions $\Psi_\ell: G \to \R$, $\Psi_\ell(g) = f(\ell ge_d)$, viewed as functions on $\Pl_d/\gop$. 
      
In the integral appearing in \eqref{eq:intsum} we replace $\Pl_d/\gop$ by the fundamental domain $\Sl \subset \Pl_d$ from \eqref{eq:S}. 
   $$ \int_{\Pl_d / \gop } \Psi_\ell(g) \, d\vol_{\Pl_d / \gop } =
   \int_\Sl \Psi_\ell(g) \, d\vol_{\Sl} = \int_\R \int_{\Fl_0(t)} \Psi_\ell(g) \, d\vol_{\Hl_r(t)}\, dt. $$ 
   The function $\Psi_\ell$ is constant on $\Hl_r(t)$. Indeed, $\Hl_r(t)$ can be identified with $a_{-t}A'N$ {\it{via}} $\hc_d: S_d \to \Pl_d$ and,
   since for $a'n \in A'N$, $a' n e_d = e_d$,
   $$ \Psi_\ell(a_{-t}a'n e_d) = f(\ell a_{-t} e_d) = f(\ell e^{\mu t/2} e_d).$$
   Using $\Fl_0(t) = \Phi_t(\Fl_0(0))$ and \eqref{eq:volhormod}, we obtain
   \begin{eqnarray*} 
   \int_{\Pl_d / \gop } \Psi_\ell (g)\, d\vol_{\Pl_d / \gop } &=&
   \int_\R f(\ell e^{\mu t/2} e_d) \vol(\Fl_0(t)) dt \\
   &=& \vol(\Fl_0(0)) \int_\R  f(\ell e^{\frac{t}{2} \sqrt{\frac{d-1}{d}}} e_d) e^{\frac{t}{2} \sqrt{(d-1)d}} dt \\
   &=& 2\, \vol(\Fl_0(0)) \sqrt{\frac{d}{d-1}} \int_0^\infty \left( \frac{r}{\ell} \right)^d f(r e_d) \frac{dr}{r}
   \end{eqnarray*}
   {\it{via}} the substitution $r = \ell e^{\frac{t}{2} \sqrt{\frac{d-1}{d}}}$.
  The function $f$ is radial, therefore,
   \begin{equation*}
   \int_{\R^d} f\, d\vol_{\R^d} = \int_0^\infty \int_{S_r(0)} f\, d\vol_{S_r(0)}\, dr = \vol_{d-1}(S_1(0)) \int_0^\infty r^{d-1} f(r e_d) dr,
   \end{equation*}
   where $S_r(0)$ denotes the $(n-1)$-dimensional Euclidean sphere of radius $r$. Combining the two previous identities and using $\omega_d = \vol_{\R^d}(B_1(0)) = \frac{1}{d} \vol_{d-1}(S_1(0))$ leads to
   $$ \int_{\Pl_d / \gop } \Psi_\ell (g)\, d\vol_{\Pl_d / \gop } = \frac{1}{\ell^d} \frac{2\, \vol(\Fl_0(0))}{d\, \omega_d} \sqrt{\frac{d}{d-1}} \int_{\R^d} f d\vol_{\R^d}. $$
    We plug this into \eqref{eq:intsum} and write the integral
    $\int_{\R^d} f d\vol_{\R^d}$ as the value $\widehat f(0)$ of the Fourier transform
    $$ \widehat f(\xi) = \int_{\R^d} f(x) e^{-2\pi i \langle x,\xi \rangle} dx. $$
    We obtain that
    \begin{equation} \label{eq:crucintident} 
    \int_{\Ml_d} S_f\, d\vol_{\Ml_d} = f(0)\vol(\Ml_d) + \alpha(d) \,\frac{2\, \vol(\Fl_0(0))}{\omega_d \, \sqrt{(d-1)d}} \widehat f(0) 
    \zeta(d). 
    \end{equation}
    Now we use the Poisson summation formula, namely,
    $$ \sum_{x \in \Lambda} f(x) = \frac{1}{\covol(\Lambda)}
    \sum_{\xi \in \Lambda^*} \widehat f(\xi), $$
    for any Schwartz function
    $f \in C^\infty(\R^d)$, any lattice $\Lambda \subset \R^d$ and dual lattice
    $$ \Lambda^* = \{ \xi \in \R^d \mid \langle x,\xi \rangle
    \in \Z , \quad \forall\, x \in \Lambda \}. $$
    We obtain
    $$ S_f(g) = \sum_{x \in g \Z^d} f(x) = \sum_{\xi \in (g^{-1})^\top
    \Z^d} \widehat f(\xi) = S_{\widehat f}((g^{-1})^\top). $$
    The automorphism $G \ni g \mapsto (g^{-1})^\top \in G$ preserves the measure $dg$ and preserves $\Gamma$ as a set and, therefore, we have
     \begin{equation} \label{eq:poisumcons} 
    \int_{G/\Gamma} S_f(g)\, dg = \int_{G/\Gamma} S_{\widehat f}((g^{-1})^\top)\, dg = \int_{G/\Gamma} S_{\widehat f}(g)\, dg. 
    \end{equation}
    We chose $f \in C^\infty(\R^d)$ to be a \emph{radial} Schwartz function, hence $\widehat f \in C^\infty(\R^d)$ is a
    \emph{radial} Schwartz function and \eqref{eq:poisumcons} descends to the integral identity
    $$ \int_{\Pl_d/\Gamma} S_f(g)\, dg = \int_{\Pl_d/\Gamma} S_{\widehat f}(g)\, dg. $$
   This identity and
    \eqref{eq:crucintident} imply that
    \begin{equation} \label{eq:crucrelfhatf}
    f(0)\vol(\Ml_d) + \alpha(d)\, \frac{2\, \vol(\Fl_0(0))}{\omega_d \, \sqrt{(d-1)d}} \widehat f(0) \zeta(d) 
    = \widehat f(0)\vol(\Ml_d) + \alpha(d)\, \frac{2\, \vol(\Fl_0(0))}{\omega_d \, \sqrt{(d-1)d}} f(0) \zeta(d),
    \end{equation}
    since we have for radial functions $\hat{\hat{f}} = f$,
    by the Fourier inversion formula
    $$ f(x) = \int_{\R^d} \widehat f(\xi) e^{2\pi i \langle x,\xi \rangle} d\xi. $$
    Since we can easily find
    a radial Schwartz function $f \in C^\infty(\R^d)$ with $f(0) = 0$
    and
    $$ \widehat f(0) = \int_{\R^d} f\, d\vol_{\R^d} = 1, $$
    the identity \eqref{eq:crucrelfhatf} simplifies in this case to
    $$ \alpha(d) \, \frac{2\, \vol(\Fl_0(0))}{\omega_d \, \sqrt{(d-1)d}} \zeta(d) =
    \vol(\Ml_d), $$
    which finishes the proof of the proposition.
\end{proof}  
 
\begin{example}
  For $d=2$, the symmetric
  space $\Pl_2$ is isometric to the rescaled hyperbolic plane $(\H^2,\sqrt{2} \dist_{\H})$, $\Fl_0(0)$ can be identified with the line segment $i + [-1/2,1/2) \subset \H^2$ and a fundamental domain for $\SLZ$ is  
  $$ \Fl_2= \left\{ z \in \C \mid - \frac{1}{2} \le {\rm{Re}}(z) < \frac{1}{2}, |z| \ge 1 \right\} \subset \H^2. $$
 Taking the rescaling factor into account, we have
  $$ \vol(\Fl_0(0)) = \sqrt{2} \quad \text{and} \quad
  \vol(\Fl_2)= \vol(\Ml_2) = \frac{2 \pi}{3}. $$
Therefore, for $d=2$,
  $$ 2 \frac{\sqrt{2}}{2 \pi/3} = \alpha(d)\, \frac{\vol(\Fl_0(0))}{\vol(\Ml_d)} = \frac{\sqrt{d(d-1)}}{2} \frac{\omega_d}{\zeta(d)} = \frac{\sqrt{2}}{2} \frac{\pi}{\pi^2/6}, $$
  which confirms the statement of Proposition \ref{prop:siegel}.
\end{example} 

A consequence of Proposition \ref{prop:siegel} that will be useful in future arguments is the following.

\begin{prop} \label{prop:siegelhorob}
For every $\tau >0$ large enough we have
  $$\frac{\vol\left( \Hlb_{\bar \rho }(-\tau ) \right)}{\vol(\Ml_d)} \leq \frac{1}{\alpha(d)}\frac{\omega_d}{\zeta(d)} e^{-\frac{\tau }{2}\sqrt{d(d-1)}}, $$
  where $\omega_d$ is the volume of the $d$-dimensional Euclidean unit ball, $\alpha(d)$ is defined in \eqref{eq:alpha}, and $\zeta$ is the Riemann zeta function.
\end{prop}

\proof The map $\Hlb_{ \rho }(-\tau )/ \gop \to \Hlb_{\bar \rho }(-\tau )$ is onto and it contracts distances and volumes, therefore it suffices to bound $\frac{\vol\left(\Hlb_{ \rho }(-\tau )/ \gop \right)}{\vol(\Ml_d)}$.
We can write, using Lemma \ref{lem:horovolrel}, that  
\begin{multline}  
\vol\left( \Hlb_{ \rho }(-\tau )/ \gop \right)= \int_\tau^\infty e^{-\frac{t}{2}\sqrt{d(d-1)}} \vol \left(\Hl_{ \rho }(0 )/ \gop \right)\, dt
= \\ \frac{2}{\sqrt{d(d-1)}} 
e^{-\frac{\tau }{2}\sqrt{d(d-1)}} \vol \left(\Hl_{ \rho }(0 )/ \gop \right) = e^{-\frac{\tau }{2}\sqrt{d(d-1)}}\frac{1}{\alpha(d)}\frac{\omega_d}{\zeta(d)} \vol(\Ml_d), 
\end{multline}
where the last equality follows from Proposition \ref{prop:siegel}. Note that, as $\rho$ is obtained from the ray opposite to $r$ by applying a permutation, that is by applying an element of the finite Weyl group, $\rho$ is the image of $r$ by an isometry, therefore Proposition \ref{prop:siegel} is also valid for $r$ replaced by $\rho$.\endproof

\section{Lattice point counting in ellipsoids}
\label{sec:latcount}

In this section we provide a brief survey of results on counting of lattice points in $\Z^d$, $d \ge 2$, contained in dilations of $d$-dimensional convex bodies.

\subsection{General case of convex bodies}
Let $\Bl^d \subset \R^d$ be a convex body with smooth boundary with bounded, nowhere vanishing Gaussian curvature.

Given
$$ N_0(\Bl^d,R) = \# (R \Bl^d \cap \Z^d) = \vol(\Bl^d) R^d + E_0(\Bl^d,R) ,$$
we are interested in the asymptotics of the error term $E_0(\Bl^d,R)$. The best estimate for $d \ge 3$, under these general assumptions on $\Bl^d$, is due to
J. Guo \cite{Guo} (improving earlier results by  W. M\"uller \cite{Mu}) and it states that
\begin{equation} \label{eq:Bod34d} 
E_0(\Bl^d,R) = \begin{cases} O(R^{231/158}) & \text{if $d=3$,} \\
O(R^{61/26}) &\text{if $d=4$,} \\ O\left( R^{d-2+\frac{d^2+3d+8}{d^3+d^2+5d+4}} \right) & \text{if $d \ge 5$.} \end{cases} 
\end{equation}
The best known bound in dimension $d=2$, due to Huxley \cite{Hux}, is that $E_0(\Bl^2,R) = O(R^{\frac{131}{208}}(\log R)^{\frac{18627}{8320}})$.

\subsection{Case of ellipsoids}
In this paper, we are mainly interested in the case where $\Bl^d$ is a $d$-dimensional ellipsoid $\El^d$, centered at the origin. Such an ellipsoid $\El^d$ can be written as $Q^{-1}([0,1])$, where $Q: \R^d \to [0,\infty)$ is a positive definite quadratic form defined {\it{via}} a symmetric $d \times d$ matrix $M_Q$. We call the ellipsoid $\El^d$ {\em rational} if there exists a factor $\lambda > 0$ such that all the entries of $\lambda M_Q$ are integers, otherwise we call $\El^d$ {\em irrational}. 

Bentkus and G\"otze \cite{BG97,BG99} proved that for all $d \ge 9$,
$E_0(\El^d,R) = O(R^{d-2})$, and
$E_0(\El^d,R) = o(R^{d-2})$ if $\El^d$ is irrational. The former estimate is optimal for rational ellipsoids since by \cite[page 17]{IKKN} $E_0(\El^d,R) = \Omega(R^{d-2})$ for all rational ellipsoids of dimension $d \ge 3$ \footnote{For an arbitrary function $f$ and a positive function $g$, we write $f(R) = \Omega_+(g(R))$ if $\limsup_{R \to \infty} f(R)/g(R) > 0$ and $f(R) = \Omega_-(g(R))$ if $\liminf_{R \to \infty} f(R)/g(R) < 0$.  $f(R) = \Omega(g(R))$ means that at least one of these two asymptotics holds. This means also that $f(R) \neq o(g(R))$.}. In \cite{G04}, G\"otze extended the above mentioned estimate to $d \ge 5$. More precisely, given $\Lambda \ge \lambda > 0$ maximal and minimal eigenvalues of 
$M_Q$, G\"otze proved that
\begin{equation} \label{eq:Goetze}
|E_0(\El^d,R)| \le C(d) \left( \frac{\Lambda}{\lambda} \right)^{d+1} 
\left( \frac{1}{\lambda} \right)^{\frac{d-2}{2}} \left( 1 + \log \frac{\Lambda}{\lambda} \right) R^{d-2}, 
\end{equation}
for all $R \ge 2 \sqrt{\lambda}$, with constants $C(d) > 0$ only depending on $d$. 

The best error term estimate for the $3$-dimensional unit ball $\Dl^3$ centered at the origin, due to Heath-Brown \cite{HB}, is that for every $\varepsilon >0$
$$ E_0(\Dl^3,R) = O(R^{21/16+\varepsilon}).$$
Chamizo, Crist{\'o}bal and Ubis \cite{CHCU09} extended the above error estimate to all rational $3$-dimensio\-nal ellipsoids. 
The best existing estimate from below of the error term for $3$-dimensional rational ellipsoids \cite{LTs02} is $E_0(\El^3,R) = \Omega_\pm(R(\log R)^{1/2})$.  

\subsection{Connection with counting of primitive lattice points}

For our purposes, we need corresponding results for the error terms of the counting of primitive lattice points $\widehat \Z^d$.
We define
$$ N_1(\Bl^d,R) = \# (R \Bl^d \cap \widehat \Z^d) = \frac{\vol(\Bl^d)}{\zeta(d)} R^d + E_1(\Bl^d,R), $$
for an arbitrary body $\Bl^d \subset \R^d$. The following fact is well-known.

\begin{prop} \label{prop:latttoprimlatt} Let $\Bl^d$ be a starlike body in $\R^d$, ith $d \ge 3$, and let $E_j(\Bl^d,R)$, $j=0,1$, be the the error terms of the corresponding full and primitive lattice point counting problems, respectively. Let $1 < \alpha < d$, $\beta \ge 0$ and $j \in  \{0,1\}$. Then an estimate of the form
\begin{equation} \label{eq:Ejest} 
|E_j(\Bl^d,R)| \le C_j R^\alpha (\log R)^\beta \qquad \text{for all $R \ge R_0 \ge 2$} 
\end{equation}
leads to a corresponding result
\begin{equation} \label{eq:E1-jest}
|E_{1-j}(\Bl^d,R)| \le C_{1-j} R^\alpha (\log R)^\beta \qquad \text{for all $R \ge R_0$}, 
\end{equation}
where $C_{1-j}$ depends only on $\alpha$, $\beta$, $d$, $\vol(\Bl^d)$ and $C_j$. 
Similarly, we have for any $1 < \alpha \le d$,
\begin{equation} \label{eq:Ejo-est} 
E_0(\Bl^d,R) = o(R^\alpha) \quad \Longleftrightarrow \quad E_1(\Bl^d,R) = o(R^\alpha). 
\end{equation}
\end{prop}

For the sake of completeness, we provide a short proof of this result in Appendix \ref{app:primlattriem}.

\subsection{A mean square bound for primitive lattice points in ellipsoids}
We recall an inequality due to W. Schmidt \cite[p. 518]{Schm60}, based on Rogers' second moment formula \cite{Rog}. The version reproduced here is from Kelmer-Yu \cite[formula (0.2)]{KY21}. 

Given a lattice $\Lambda$ and a Borel set
$\Bl$ in $\R^d$, let $\Lambda_{\rm{pr}}$ be the subset of
primitive vectors of $\Lambda$ and let
$$ D(\Lambda,\Bl) = \left| \frac{\zeta(d) \# (\Lambda_{\rm{pr}} \cap \Bl)}{\vol(\Bl)} - 1 \right|. $$

Schmidt's inequality states that for $d \ge 3$
\begin{equation}\label{eq:schmidt}
\int_{\gd / \gad} D(\Lambda,\Bl)^2 d\nu_d(\Lambda) \le 
\frac{2\zeta(d)}{\vol(\Bl)},
\end{equation}

where $\nu_d$ is the unique bi-invariant probability measure on $\gd / \gad$.

\smallskip

Using the identity
$$
N_1(\El_{Q_0 \cdot g},R) = \# \left\{ v \in \widehat \Z^d \vert (Q_0 \cdot g)(v) \le R^2 \right\} = \# (g \widehat \Z^d \cap B_R(0)),
$$
where $B_R(0)$ denotes the $d$-dimensional ball of radius $R$ centered at the origin, we obtain
$$ | E_1(\El_{Q_0\cdot g},R) | = \left| N_1(\El_{Q_0\cdot g},R) - \frac{\vol(\El_{Q_0\cdot g})}{\zeta(d)} R^d \right| = \frac{\vol(B_R(0))}{\zeta(d)} D(g \Z^d,B_R(0)).  $$
This and the $\SOd$-invariance allows to conclude that
\begin{multline} \label{eq:meansqbd}
\oint_{\Ml_d} | E_1(\El_Q,R) |^2 d\vol_{\Ml_d}(\bar Q) = 
\frac{\vol(B_R(0))^2}{\zeta(d)^2} \int_{\gd / \gad} |D(\Lambda,B_R(0))|^2 d\nu(\Lambda) \le
\frac{2 \omega_d}{\zeta(d)} \cdot R^d. 
\end{multline}

In the case $d=2$ the second moment formula provided by Rogers in \cite[Theorem 5]{Rog} is incorrect, as explained in \cite[page 61]{Fairchild}. This problem was addressed by Schmidt in \cite{Schm60}, from Section 6 onwards, where he provided a corrected version of the formula. The version reproduced here follows from Kelmer-Yu \cite[Theorem 1]{KY21} for $n=1$, who state that this particular case is a consequence of \cite[Proposition 2.10]{KM12}. Indeed, Theorem 1 in \cite{KY21} applied to the characteristic function of a Borel set in $\R^2$ implies inequality \eqref{eq:schmidt} for $d=2$, with the right hand side $\frac{4\zeta(2)}{\vol(\Bl)}$. Therefore \eqref{eq:meansqbd} is also true for $d=2$, with the right hand side $\frac{4 \omega_2}{\zeta(2)} \cdot R^2$.

The above mean square bounds play a central part in the proof of the equidistribution results given in Section \ref{sec:effeqdisthorospheres}.

\section{$\Gamma$-orbit counting in the symmetric space $\Pl_d$}
\label{sec:orbitcounting}

Our first goal is to translate the counting of primitive lattice points
$\widehat \Z^d$ in dilations of $d$-dimensional ellipsoids into counting of $\Gamma$-orbits in increasing subsets of the
symmetric space $\Pl_d$. As before, every quadratic form $Q \in \Pl_d$
gives rise to an ellipsoid $\El_Q = \{ x \in \R^d \mid Q(x) \le 1 \}$
of volume $\omega_d = \vol_{\R^d}(B_1(0))$.  Using the identification
$$ \Gamma / \gop \to {\widehat \Z}^d, \quad \gamma \gop \mapsto \gamma e_d, $$
we obtain
\begin{equation*}
N_1(\El_Q,R) = \# \{x \in {\widehat \Z}^d \mid Q(x) \le R^2 \} 
= \sum_{\gamma \in \Gamma/\gop } \chi_{(0,R^2]}(Q(\gamma e_d)) = \sum_{\gamma \in \Gamma/\gop } \chi_{(-\infty,T]}(f_r(Q \cdot \gamma)), 
\end{equation*}
with $f_r: \Pl_d \to \R$ the Busemann function given in \eqref{eq:busem}
and
$$ T = T(R) = 2 \sqrt{\frac{d}{d-1}} \log R. $$
Before we proceed, we introduce some useful notation.

\begin{notation}\label{notat:chimney}
Henceforth we will refer to the fundamental domain of the $\gop$-action on $\Pl_d$, defined as in \eqref{eq:S} by
$$ \Sl = \bigcup_{t \in \R} \Fl_0(t), $$
as \emph{the chimney},
and to 
\begin{equation} \label{eq:ST} 
\Sl_T = \bigcup_{t \in (-\infty,T]} \Fl_0(t) 
\end{equation}
as the \emph{chimney truncated at level} $T \in \R$ or, simply, as a \emph{truncated chimney}.
\end{notation}

In what follows we need the value of the volume of a truncated chimney.
\begin{equation}\label{eq:volst}
\vol(\Sl_T) = \int_{-\infty}^T \vol(\Fl_0(t)) dt = \int_{-\infty}^T e^{\frac{t}{2}\sqrt{(d-1)d}} \vol(\Fl_0(0)) dt
\end{equation}
$$
 = \frac{2}{\sqrt{(d-1)d}} \vol(\Fl_0(0)) \, e^{\frac{T}{2}\sqrt{(d-1)d}}.  
$$

We continue the calculation preceding Notation \ref{notat:chimney}, and obtain
\begin{multline} \label{eq:N1ref}
N_1(\El_Q,R) = \sum_{\gamma \in \Gamma} \chi_\Sl(Q \cdot \gamma) \chi_{(-\infty,T]}(f_r(Q \cdot \gamma)) = \sum_{\gamma \in \Gamma} \chi_{\Sl_T}(Q \cdot \gamma) = \alpha(d) \sum_{\gamma \in P\Gamma} \chi_{\Sl_T}(Q \cdot \gamma)\\
 =\alpha(d) \, 
\# \{ \gamma \in P\Gamma \mid Q \cdot \gamma \in \Sl_T \}=\alpha(d) \sigma (Q) \, 
\# \left( \{ Q \cdot \gamma \mid \gamma \in \Gamma \} \cap \Sl_T \right),
\end{multline}
where $\sigma (Q)$ is the order of the stabilizer of $Q$ in $P\Gamma$.

Using this connection between counting problems, and the results discussed in Section \ref{sec:latcount}, we obtain

\begin{thm} \label{thm:orbitcount}
Given a point $Q \in \Pl_d$, we have the following asymptotics of the number of orbit points  $Q \cdot \Gamma \subset \Pl_d$ contained in the truncated chimney $\Sl_T$ defined in \eqref{eq:ST}:
$$ \frac{\sigma(Q) \# \{ Q \cdot \gamma \in \Sl_T \mid \gamma \in \Gamma \}}{\vol(\Sl_T)/\vol(\Ml_d)} = 1 + \begin{cases} 
  O\left(T^{\frac{18627}{8320}} \exp\left(-\frac{285}{416 \sqrt{2}}T\right)\right) & \text{if $d=2$,}\\
 O\left( \exp\left(-\frac{243}{158\sqrt{6}}T\right) \right) & \text{if $d=3$,} \\
 O\left( \exp\left(-\frac{43\sqrt{3}}{104}T\right) \right) & \text{if $d=4$,} \\
 O\left( \exp\left(-\sqrt{\frac{d-1}{d}}T\right) \right) & \text{if $d\ge5 $.}
  \end{cases}
  $$
\end{thm}

\begin{proof}
  Since $\vol(\El_Q) = \omega_d$ for all $Q \in \Pl_d$, we know from Section \ref{sec:latcount} that
  \begin{equation} \label{eq:asympsN1}
  N_1(\El_Q,R) - \frac{\omega_d}{\zeta(d)} R^d = \begin{cases} 
  O\left(R^{\frac{131}{208}} (\log R)^{\frac{18627}{8320}}\right) & \text{if $d=2$,}\\
 O\left(R^{\frac{231}{158}}\right) & \text{if $d=3$,} \\
 O\left(R^{\frac{61}{26}}\right) & \text{if $d=4$,} \\
 O\left(R^{d-2}\right) & \text{if $d\ge5 $.}
  \end{cases}
 \end{equation}
We also know from \eqref{eq:N1ref} that
\begin{equation} \label{eq:N1refag} 
N_1(\El_Q,R) = \alpha(d) \sigma (Q) \, \# \{ Q \cdot \gamma \in \Sl_T \mid \gamma \in \Gamma \}, \quad \text{with $T = 2 \sqrt{\frac{d}{d-1}}\log(R)$.} 
\end{equation}
On the other hand, using \eqref{eq:volst} and Proposition \ref{prop:siegel}, we obtain
\begin{equation} \label{eq:volcalc}
\frac{\omega_d}{\zeta(d)} R^d = \frac{\omega_d}{\zeta(d)} e^{\frac{\sqrt{(d-1)d}}{2} T} = \frac{\omega_d}{\zeta(d)}\frac{\sqrt{(d-1)d}}{2} \frac{\vol(\Sl_T)}{\vol(\Fl_0(0)) }= 
\alpha(d) \, \frac{\vol(\Sl_T)}{\vol(\Ml_d)}.
\end{equation}
The theorem follows now by combining \eqref{eq:asympsN1}, \eqref{eq:N1refag} and\eqref{eq:volcalc}.
\end{proof}

\begin{example}
  Under the identification \eqref{eq:homhypplane} of $(\Pl_2,\dist_{\Pl_2})$ with the rescaled hyperbolic plane $(\H^2,\sqrt{2}\dist_\H)$, the truncated chimney $\Sl_T  \subset \Pl_2$ corresponds to the semi-infinite strip
  $$ \Sl_y = \left\{ z \in \C \Biggm| -\frac{1}{2} \le {\rm{Re}}(z) \le \frac{1}{2},
    {\rm{Im}}(z) \ge y \right\}, \quad \text{with $y = e^{-T/\sqrt{2}}$.} $$
  Since $\area(\Sl_y) = \frac{2}{y}$ and
  $\area(\Ml_2) = \frac{2\pi}{3}$, Theorem \ref{thm:orbitcount} states
  that, for every $z \in \H$, we have that, as $y \to 0$,
  $$ \frac{ |{\mathrm Stab}_\Gamma (z)|\, \cdot\, \# \left\{ \gamma \cdot z \in \Sl_y \mid \gamma \in \SLZ \right\} }
  {3/(\pi y)} = 1 +
  O\left( (\log y)^{\frac{18627}{8320}} y^{\frac{285}{416}} \right).$$
\end{example}

Finally, we present the counting result for horospheres mentioned in the
Introduction.

\begin{thm} \label{thm:orbitcount2}
Given a point $Q \in \Pl_d$, we have the following asymptotics of the number of horospheres in the orbit $\Hl_r \Gamma$ that are at distance at most $T$ from $Q$:
  $$ \frac{ \# \{ \Hl \in \Hl_r \Gamma \mid \Hl \cap B(Q, T)\neq \emptyset \}}
  {\kappa_d\, e^{\frac{T}{2}\sqrt{(d-1)d}}/\vol(\Ml_d)} = 1 + \begin{cases} 
  O\left(T^{\frac{18627}{8320}} \exp\left(-\frac{285}{416 \sqrt{2}}T\right)\right) & \text{if $d=2$,}\\
 O\left( \exp\left(-\frac{243}{158\sqrt{6}}T\right) \right) & \text{if $d=3$,} \\
 O\left( \exp\left(-\frac{43\sqrt{3}}{104}T\right) \right) & \text{if $d=4$,} \\
 O\left( \exp\left(-\sqrt{\frac{d-1}{d}}T\right) \right) & \text{if $d\ge5 $,}
  \end{cases}
$$
where $\kappa_d =\frac{\alpha (d)}{\sqrt{(d-1)d}}\, \vol(\Fl_0(0)) = \frac{\vol(\Ml_d)}{2}\,  \frac{\omega_d}{\zeta(d)}$.
\end{thm}

\begin{proof} We have that 
\begin{multline*}
N_1(\El_Q,R) = \# \{ \gamma \gop \in \Gamma /\gop \mid Q(\gamma e_d) \le R^2 \} 
= \# \{ \gamma \gop \in \Gamma /\gop \mid f_r (Q\cdot \gamma) \le T \}
\\
= \# \{ \gop \gamma  \in \gop \backslash \Gamma \mid f_{r\cdot \gamma} (Q ) \le T \},
\end{multline*}
with $f_r: \Pl_d \to \R$ the Busemann function given in \eqref{eq:busem}
and $ T = 2 \sqrt{\frac{d}{d-1}} \log R.$ It follows that
\begin{multline*}
N_1(\El_Q,R) = \\ \# \{ \gop \gamma  \in \gop \backslash \Gamma \mid \dist_{\Pl_d} (Q, \Hl_{r\cdot \gamma}) \le T \}
=\alpha (d) \# \{ \gop \gamma  \in \gop \backslash P\Gamma \mid \dist_{\Pl_d} (Q, \Hl_{r\cdot \gamma}) \le T \}\\
=2 \# \{ P\Gamma_0 \gamma  \in P\Gamma_0 \backslash P\Gamma \mid \dist_{\Pl_d} (Q, \Hl_{r\cdot \gamma}) \le T \}
= 2 \# \left[ \Hl_r \Gamma \cap B(Q,T)\right].  
\end{multline*}

The factor $2$ in the last but one line appears because if $d$ is even then $\alpha (d)=2$ and $\gop $ coincides with $P\Gamma_0$, while if $d$ is odd then $\alpha (d)=1$, $\Gamma = P\Gamma$,  $\Gamma_0 = P\Gamma_0$, and $\gop$ is a normal subgroup of index two in $\Gamma_0$ (therefore two cosets of $\gop$ map to the same coset of $\Gamma_0$). 

The rest of the argument follows as in the proof of Theorem \ref{thm:orbitcount}.    
\end{proof}

\section{Effective equidistribution of expanding horospheres}
\label{sec:effeqdisthorospheres}

\subsection{Average estimate of the error term}
Recall that $\Fl_d \subset \Pl_d$ denotes a fundamental domain of the
$\Gamma$-right action on $\Pl_d$. Every function $f \in C(\Ml_d) \cap L^2(\Ml_d)$ can be lifted to a $\Gamma$-periodic function $\tilde f = f \circ \pi \in C(\Pl_d)$. Since every $Q \in \Pl_d$ represents a specific ellipsoid $\El_Q$, we can consider the function $\tilde f$ as a map associating a weight to each of these ellipsoids. Note also that, for each
$R > 0$, the map $Q \mapsto N_1(\El_Q,R)$ is $\Gamma$-invariant.

Therefore, the integral $\int_{\Ml_d} f(\bar Q) N_1(\El_Q,R) d\vol_{\Ml_d}(\bar Q)$ is well defined and, using \eqref{eq:N1ref} and the fact that $\sigma(Q) > 1$ only on a set of measure zero, we obtain
\begin{multline*}
\int_{\Ml_d} f(\bar Q) N_1(\El_Q,R) d\vol_{\Ml_d}(\bar Q) = \alpha(d) \, \sum_{\gamma \in \Gamma} \int_{\Fl_d} \tilde f(Q) \chi_{\Sl_T}(Q \cdot \gamma) d\vol_{\Pl_d}(Q) \\ = 
\alpha(d) \, \int_{\Pl_d} \tilde f(Q) \chi_{\Sl_T}(Q) d\vol_{\Pl_d}(Q) = \alpha(d) \, \int_{-\infty}^T \int_{\Fl_0(t)} \tilde f(Q) d\vol_{\Hl_r(t)}(Q) dt
\end{multline*}
for $T = 2 \sqrt{\frac{d}{d-1}} \log(R)$. We now apply Lemma \ref{lem:horovolrel}, \eqref{eq:volhormod}, and obtain
$$
\oint_{\Ml_d} f(\bar Q) N_1(\El_Q,R) d\vol_{\Ml_d}(\bar Q) = \frac{\alpha(d) \, \vol(\Fl_0(0))}{\vol(\Ml_d)} \int_{-\infty}^T e^{\frac{t}{2}\sqrt{(d-1)d}}
\oint_{\Fl_0(t)} \tilde f (Q) d\vol_{\Hl_r(t)}(Q) dt.
$$
Using Proposition \ref{prop:siegel}, the equality can be rewritten as
\begin{equation} \label{eq:unifcount}
\oint_{\Ml_d} f(\bar Q) N_1(\El_Q,R) d\vol_{\Ml_d}(\bar Q) = \frac{\sqrt{(d-1)d }\, \omega_d}{2 \zeta(d)}
 \int_{-\infty}^T e^{\frac{t}{2}\sqrt{(d-1)d}}
\oint_{\Fl_0(t)} \tilde f (Q) d\vol_{\Hl_r(t)}(Q) dt.
\end{equation}
This identity, together with the $L^2$-integral error estimate \eqref{eq:meansqbd}, are crucial ingredients in the proof of Proposition \ref{prop:equihoroball0} below.

\begin{notation}
  For $\tilde f \in C(\Pl_d)$, we define
  \begin{equation} \label{eq:F_f} 
  F_{\tilde f}(t) = \oint_{\Fl_0(t)} \tilde f (Q) d\vol_{\Hl_r(t)} (Q). 
  \end{equation}
\end{notation}

\begin{prop} \label{prop:equihoroball0} 
Let $d \ge 2$, $f \in C(\Ml_d) \cap L^2(\Ml_d)$ and $\tilde f = f \circ \pi \in C(\Pl_d)$. We have
\begin{equation} \label{eq:esthoroball0} 
\left\vert \int_{-\infty}^T
    e^{\frac{t}{2}\sqrt{(d-1)d}} \left( F_{\tilde f}(t) -
      \oint_{\Ml_d} f (\bar Q ) d\vol_{\Ml_d}(\bar Q) \right) dt \right\vert  
 \le C_d \, \Vert f \Vert_{\Ml_d} \, 
 e^{\frac{T}{4}\, \sqrt{(d-1)d}},
\end{equation}
for all $T \in \R$, where $\Vert f \Vert_{\Ml_d}^2 = \oint_{\Ml_d} f^2 d\vol_{\Ml_d}$ and
$$ C_d = 2 \sqrt{\frac{4 \zeta(d)}{d(d-1) \omega_d}}. $$
\end{prop}

Note that for $d \ge 3$, $C_d = 2 \sqrt{\frac{2 \zeta(d)}{d(d-1) \omega_d}}.$

\begin{proof}
Cauchy-Schwarz and the inequality in \eqref{eq:meansqbd} for $d\ge 2$ yield
  \begin{multline} \label{eq:complvolform}
  \left\vert \oint_{\Ml_d} f(\bar Q)N_1(\El_Q,R) d\vol_{\Ml_d}(\bar Q) - \frac{\omega_d}{\zeta(d)} R^d \oint_{\Ml_d} f(\bar Q)
     d\vol_{\Ml_d}(\bar Q) \right\vert \\
     = \left\vert \oint_{\Ml_d} f(\bar Q) E_1(\El_Q,R) d\vol_{\Ml_d}(\bar Q) \right\vert \le \Vert f \Vert_{\Ml_d} \cdot
     \left( \oint_{\Ml_d} | E_1(\El_Q,R) |^2 d\vol_{\Ml_d}(\bar Q)
     \right)^{\frac{1}{2}} \\ \le \sqrt{\frac{4\omega_d}{\zeta(d)}}
     \,  \Vert f \Vert_{\Ml_d} \, R^{\frac{d}{2}}.
  \end{multline}
 Since $R = e^{\frac{T}{2}\sqrt{\frac{d-1}{d}}}$, we have
  \begin{equation} \label{eq:volrdandelse} \frac{\omega_d}{\zeta(d)}
    R^d \oint_{\Ml_d} f(\bar Q) d\vol_{\Ml_d}(\bar Q) =
    \frac{\sqrt{(d-1)d}\, \omega_d}{2 \zeta(d)} \int_{-\infty}^T e^{\frac{t}{2}\sqrt{(d-1)d}}
    \oint_{\Ml_d} f(\bar Q)d\vol_{\Ml_d}(\bar Q) dt.
  \end{equation}
Plugging \eqref{eq:unifcount} and \eqref{eq:volrdandelse} into
  \eqref{eq:complvolform} and dividing by $\frac{\sqrt{(d-1)d}\, \omega_d}{2 \zeta(d)}$, we obtain
\begin{equation*}    
\left\vert \int_{-\infty}^T e^{\frac{t}{2} \sqrt{(d-1)d}} \left(
      F_{\tilde f}(t) - \oint_{\Ml_d} f(\bar Q) d\vol_{\Ml_d} (\bar Q) \right) dt
    \right\vert \le 2 \sqrt{\frac{4 \zeta(d)}{d(d-1) \omega_d}}
\, \Vert f \Vert_{\Ml_d}\, e^{\frac{T}{4}\sqrt{(d-1)d}}.    
\end{equation*}
%The proposition follows with $C_d = \sqrt{\frac{2 \omega_d \vol(\Ml_d)}{\zeta(d)}}$.
  
\end{proof}

\subsection{Proof of the main results} Now we make use of the results in Appendix \ref{app:realline} to prove our main results stated in the
Introduction. Let us first recall Theorem \ref{thm:main1}, namely an
effective equidistribution of the expanding horospheres $\Hl_r(t)$ modulo $\gad$
for specific values $t$ within shrinking intervals $[T,T_1]$ as
$T \to \infty$. We reformulate this result ``upstairs'' in $\Pl_d$ in
terms of the $\gad$-periodic lift of a continuous function on
$\Ml_d$.

\begin{thm} \label{thm:main1again} 
Let $d \ge 2$, $f \in C(\Ml_d) \cap L^2(\Ml_d)$ and $\tilde f = f \circ \pi \in C(\Pl_d)$. For every $\varepsilon >0$ and $\kappa>0$ there exists $T_0 = T_0(d,\varepsilon,\kappa,\Vert f \Vert_{\Ml_d})$
and $C_0 = C_0(d,\kappa,\Vert f \Vert_{\Ml_d})$
such that the following holds: In every interval 
$$ I_T = \left[ 
T,
T+ C_0 e^{-\varepsilon T}
\right]
$$
%\frac{4C_d\Vert f \Vert_{L^2(\Ml_d)}}{\kappa} e^{-\varepsilon T} 
with $T \ge T_0$, there exists $t \in I_T$ such that
$$ \left\vert F_{\tilde f}(t) - \oint_{\Ml_d} f d\vol_{\Ml_d} \right\vert
 < \kappa e^{-\left( \frac{1}{4}\sqrt{(d-1)d} \right) t + \varepsilon t}, $$
  where $F_{\tilde f}$ is as defined in \eqref{eq:F_f}.
\end{thm}

\begin{proof}[Proof of Theorem \ref{thm:main1again}]
Let 
  $$ g(t) = F_{\tilde f}(t) - \oint_{\Ml_d} f\, d\vol_{\Ml_d}. $$
Proposition \ref{prop:equihoroball0} tells us that $g \in C(\R)$ satisfies
  \begin{equation} \label{eq:gprop}
  \left\vert \int_{-\infty}^T e^{\alpha t} g(t) dt \right\vert \le C_d \Vert f \Vert_{\Ml_d} e^{\beta T} \quad \text{for all $T \in \R$}
  \end{equation}
  with 
  \begin{equation} \label{eq:albet}
  \alpha = \frac{1}{2} \sqrt{(d-1)d} \quad \text{and} \quad
  \beta = \frac{1}{4} \sqrt{(d-1)d} = \frac{\alpha}{2}.
  \end{equation}
We apply Proposition \ref{prop:intexpconseq} with $\tilde C = C_d \Vert f \Vert_{\Ml_d}$ and
  conclude that, for every interval $I_T = [T,T+C_0e^{-\varepsilon T}] \subset [T_0,\infty)$ 
  with $T_0$ in \eqref{eq:T_0} and $C_0 = \frac{4 \tilde C}{\kappa} = \frac{4 C_d \Vert f \Vert_{\Ml_d}}{\kappa}$, there exists $t \in I_T$ with
  $$ |g(t)| = \left\vert F_{\tilde f}(t) - \oint_{\Ml_d} f\, d\vol_{\Ml_d}
  \right\vert \le \kappa e^{ -(\alpha-\beta) t + \varepsilon t} = \kappa e^{-\left( \frac{1}{4}\sqrt{(d-1)d} \right) t + \varepsilon t}. $$ 
\end{proof}

Finally, we recall Theorem \ref{thm:main2} and present its proof:

\begin{thm} \label{thm:main2again} Let $d \ge 2$, $f \in C^1(\Ml_d) \cap L^2(\Ml_d)$ and $\tilde f = f \circ \pi \in C(\Pl_d)$. Then 
there exists a constant $T_d$, depending only on the dimension $d$, such that for all $t \ge T_d$ we have
 $$
    \left\vert F_{\tilde f}(t) - \oint_{\Ml_d} f\, d\vol_{\Ml_d} \right\vert
    \le \left( C_d \Vert f \Vert_{\Ml_d} + 4 \Vert \grad f \Vert_\infty
    \right) e^{ - \left( \frac{1}{8} \sqrt{(d-1)d} \right) t},
 $$
where $C_d$ is the constant from Proposition \ref{prop:equihoroball0}, $\Vert f \Vert_{\Ml_d}^2 = \oint_{\Ml_d} f^2\, d\vol_{\Ml_d}$, and $F_{\tilde f}$ is as defined in \eqref{eq:F_f}.
\end{thm}

\begin{proof}[Proof of Theorem \ref{thm:main2again}]
We can assume that $\Vert \grad f \Vert_\infty < \infty$ since, otherwise, there is nothing to prove. Let
 $$ g(t) = F_{\tilde f}(t) - \oint_{\Ml_d} f\, d\vol_{\Ml_d}. $$
 Then $g$ is Lipschitz continuous with Lipschitz constant $\Vert \grad f \Vert_\infty$, since
 $$ |g(t_2) - g(t_1)| \le \oint_{\Fl_0(t_1)}
 \left\vert \tilde f \circ \Phi_{t_2-t_1} - \tilde f \right\vert \,
 d\vol_{\Fl_0(t_1)} $$
 and
 $$ \left\vert \tilde f \circ \Phi_t(Q) - \tilde f(Q) \right\vert \le |t| \,
 \Vert \grad f \Vert_\infty \quad \text{for all $t \in \R$ and $Q \in \Ml_d$.}
 $$
 Moreover, as in the proof of Theorem \ref{thm:main1again}, $g$
 satisfies \eqref{eq:gprop} with the constants $\alpha,\beta$ chosen
 as in \eqref{eq:albet}. We apply Corollary \ref{cor:intexpconseq} and obtain
 \begin{equation*}
 |g(t)| = \left\vert F_{\tilde f}(t) - \oint_{\Ml_d} f\, d\vol_{\Ml_d}
 \right\vert \le (C_d\Vert f \Vert_{\Ml_d} + 4 \Vert \grad f \Vert_\infty)
 e^{- \left( \frac{1}{8} \sqrt{(d-1)d} \right) t}
 \end{equation*}
 for all
 $$ t \ge T_d = \frac{2}{\alpha-\beta} \log(\alpha+\beta) = \frac{8}{\sqrt{(d-1)d}} \log\left( \frac{3}{4}\sqrt{(d-1)d} \right). $$
 This finishes the proof of the theorem.
 \end{proof}

\begin{remarks}\label{rem:comparison2}
In the case $d=2$, Theorems \ref{thm:main1again} and \ref{thm:main2again} compare with existing results as follows.
\begin{enumerate}
\item The exponent in Theorem \ref{thm:main1again} is the same as the one obtained by Zagier \cite{Zag} and Sarnak \cite{Sarnak}, after proper renormalization (see explanations about this renormalization in Example \ref{ex:hypplane1}). Unlike in \cite{Zag,Sarnak}, the estimate in Theorem \ref{thm:main1again} does not cover all (large enough) values of the parameter $t$, still its interest may spring from the fact that it is obtained by a second moment formula going back to Rogers and Schmidt.     

\medskip

\item Under the Riemann Hypothesis, Zagier \cite{Zag} and Sarnak \cite{Sarnak}
    derived a much stronger result: For any
    smooth function $f$ on the unit tangent bundle $S \Ml_2$ of
    $\Ml_2$ and its $\SLZ$-periodic lift $\tilde f$ on the unit
    tangent bundle of $(\H,\dist_\H)$ and for all $\varepsilon > 0$,
    \begin{equation}\label{eq:RH}
    \int_{-1/2}^{1/2} \tilde f(X_\infty(x+iy))dx -
    \oint_{S\Ml_2} f d\vol_{S \Ml_2} = O(y^{3/4-\varepsilon}),
    \quad \text{as $y \to 0$.}
    \end{equation}
    Here $d\vol_{S \Ml_2}$ is the Liouville measure on $S \Ml_2$ and
    $X_\infty$ is the vector field of unit tangent vectors pointing towards $+ \infty \in \partial \H$. The error term in \eqref{eq:RH}
    corresponds to an error term of $O\left( e^{-\frac{3}{8}\sqrt{2} t + \varepsilon t} \right)$ in Theorems \ref{thm:main1again} and \ref{thm:main2again}.

We will discuss further connections between primitive lattice points in ellipses, $\SLZ$-orbits, and equidistribution of horocycles in $\H$ in more detail in the forthcoming paper \cite{DrP}.   
\end{enumerate}
 \end{remarks}

 \begin{remark} \label{rem:two}
    As mentioned in the paragraph preceding Notation \ref{notatf0}, we can identify $\pi(\Hl_r(t))$ with the
    quotient $\Hl_r(t)/\Gamma_0$ which, in turn, coincides with $\Hl_r(t)/\gop $ when $d$ is even, or it is obtained from $\Hl_r(t)/\gop $ by taking a quotient with the group $\two \simeq \Z_2$, defined in \eqref{eq:two}, when $d$ is odd.

Consider an arbitrary function $f \in C(\Ml_d)$ and its $\Gamma$-periodic lift
    $\tilde f = f \circ \pi \in C(\Pl_d)$. We have that 
    $$
\oint_{\Fl_0(t)}
     \tilde f \, d\vol_{\Hl_r(t)} = \oint_{\pi(\Hl_r(t))} f d\vol_{\pi(\Hl_r(t))}.   
    $$
    
The equality is obvious when $d$ is even. When $d$ is odd, $\Fl_0(t)$ is composed of a fundamental domain of $\Gamma_0$ and a translate of it by the nontrivial element of $\two$. Using this and the $\Gamma$-invariance of the function $\tilde f$ we obtain the equality between the two normalized integrals. 

Thus, if we wish to formulate Theorems \ref{thm:main1again} and \ref{thm:main2again} entirely in terms of the quotient space $\Ml_d$, we can replace $F_{\tilde f}(t)$ by $\oint_{\pi(\Hl_r(t))} f d\vol_{\pi(\Hl_r(t))}$, as is done in the introduction.
\end{remark}    

The following remark is in preparation of the next subsection.

\begin{remark} \label{rem:conint} 
The quotient $\Hl_r(t)/\gop $ is a fiber bundle
    over the space $\Ml_{d-1}$ with fibers diffeomorphic to
    $\T^{d-1} = \R^{d-1}/\Z^{d-1}$:
    \begin{equation*}
    \xymatrix@C=6mm@R=6mm{\T^{d-1} \ar[r] & \Hl_r(t)/\gop \ar[d] \\
    & \Ml_{d-1}
    }
    \end{equation*}
    Let $\Fl_{d-1} \subset \Pl_{d-1}$
    be a fundamental domain of the $\SLdmZ$-right action on
    $\Pl_{d-1}$. Using \eqref{eq:F00} and the inclusion map
    \begin{equation} \label{eq:iota} 
    \iota: \gdm \to \gd, \quad \iota(h) = \begin{pmatrix} h & 0 \\ 0 & 1
    \end{pmatrix}, 
    \end{equation}
    we have
    \begin{multline}
     \oint_{\pi(\Hl_r(t))} f d\vol_{\pi(\Hl_r(t))} = \oint_{\Fl_0(0)}
     \tilde f \circ \Phi_t\, d\vol_{\Hl_r(0)} \\ =
     \oint_{\hc_{d-1}^{-1}(\Fl_{d-1})} \int_{\left[ -\frac{1}{2},
         \frac{1}{2}\right)^{d-1}} \tilde f \left( SO(d) \cdot
        \iota(h) \begin{pmatrix} e^{-\frac{\lambda t}{2}} \Id_{d-1} &
         0 \\ e^{\frac{\mu t}{2}} x^\top & e^{\frac{\mu
             t}{2}} \end{pmatrix} \right) dx\, dh.
    \end{multline}
    The inner integral can be viewed as an integration over the fiber and,
    introducing the function
    \begin{equation} \label{eq:Fth} 
    \tilde F(t,h) = \int_{\left[
        -\frac{1}{2}, \frac{1}{2}\right)^{d-1}} \tilde f \left( SO(d)
     \cdot \iota(h) \begin{pmatrix} e^{-\frac{\lambda t}{2}}
        \Id_{d-1} & 0 \\ e^{\frac{\mu t}{2}} x^\top & e^{\frac{\mu
            t}{2}} \end{pmatrix} \right) dx, 
    \end{equation}
    we see that $\tilde F(t,\cdot)$ is $\SLdmZ$-right invariant and
    descends, therefore, to a function $F(t,\cdot)$ on the quotient
    $\Ml_{d-1}$. We end up with the identity
    \begin{equation} \label{eq:int-recursion} 
    \oint_{\pi(\Hl_r(t))}  f d\vol_{\pi(\Hl_r(t))} = \oint_{\Ml_{d-1}}
    F(t,\cdot) d\vol_{\Ml_{d-1}}. 
    \end{equation}
\end{remark}

\subsection{Improvement of main results for functions with compact support {\it{via}} truncation}

In this subsection, we only consider compactly supported functions $f \in C_c(\Ml_d)$.
By \eqref{eq:int-recursion}, 
the average $F_{\tilde f}(t)$ can be written as a normalized integral over $\Ml_{d-1}$ of the fiber integral  $F(t,\cdot)$ defined in \eqref{eq:Fth}. Since many of the dilated unipotent fibers over the base $\Ml_{d-1}$ stay high in the cusp of $\Ml_d$ and do not intersect the compact support of the given function $f$, in the normalized integral \eqref{eq:int-recursion} we can replace $\Ml_{d-1}$ by a truncated compact subset, obtained by cutting off the cusp along a sufficiently high horosphere. In particular, we can replace in both our main theorems the average $F_{\tilde f}(t)$ by an average over a fiber bundle with truncated compact base (with modifications in the constants but the same exponents). However, as $t \to \infty$, the truncated base needs to 
become larger and larger.
Theorem \ref{thm:main3again} below is a  quantitative version of these relevant geometric facts.

\smallskip

Recall that we denote by $r$ the geodesic ray defined in \eqref{eq:rt}
with associated Busemann function $f_r$ and associated horoballs $\Hlb_r(a) \subset \Pl_d$ described in \eqref{eq:horo}.

Following Notation \ref{notat:gammadpi}, given the projection $\pi: \Pl_d \to \Ml_d$, we denote by $\bar r$ the isometric projection of $r$ under $\pi $ and by $\Hlb_{\bar r}(a)$ the horoballs defined by $\bar r$ as in \eqref{eq:Hbr(a)}. Recall that for $a<0$ with $|a|$ large enough, $\Hlb_{\bar r}(a)$ is the projection of $\Hlb_{r}(a)$.

\begin{convention}
In what follows, $\rho$ denotes the geodesic ray in $\Pl_{d-1}$ defined by \eqref{eq:rhot}, as well as its image in $\Pl_{d}$ {\it{via}} the inclusion map $\iota: \gdm \to \gd$ defined in \eqref{eq:iota}. 
Let $\bar \rho$ be the projection of $\rho $ into $\Ml_{d-1}$.    
\end{convention}

We define, for $a<0$ with $|a|$ large enough, the truncated locally symmetric space
$$ \Ml_{d-1}(a) = \Ml_{d-1} \setminus \Hlbo_{\bar \rho }(a)=  \{ \bar{Q} \in \Ml_{d-1} \mid f_{\bar \rho }(\bar Q) \ge a \}. $$
A lift of $\Ml_{d-1}(a)$ in the fundamental domain $\Fl_{d-1} \subset \Pl_{d-1}$, as described in Notation \ref{notat:fld}, is then given by
$$ \Fl_{d-1}(a) = \{ Q \in \Fl_{d-1} \mid f_{\rho }(Q) \ge a \}. $$
In the special case $d=3$ and under the identification of $\Pl_2$ with $(\H,\sqrt{2}\dist_\H)$ (as in Example \ref{ex:hypplane1}) and of $\Fl_2$ with 
$\M = \{ z \in \H \mid -1/2 < {\rm{Re}}(z) < 1/2, |z| > 1 \}$,
$\Fl_2(a)$ is then identified with $\{ z \in \M \mid {\rm{Im}}(z) \le e^{-\frac{a}{\sqrt{2}}} \}$.

Correspondingly, we define a truncation of $\Hl_r (0)/\gop$, seen as a fibre bundle over $\Ml_{d-1}$, as follows:
\begin{equation} \label{eq:F0t'} 
\Fl_0(0)^{-a}=\left\{ \begin{pmatrix} g' & 0 \\ z^\top & 1 \end{pmatrix} \Biggm| g' \in \hc_{d-1}^{-1}(\Fl_{d-1}(a)), z \in [-1/2,1/2)^{d-1} \right\}.
\end{equation}

\begin{thm} \label{thm:main3again}
Let $d\geq 3$ and $\alpha > \frac{1}{2\sqrt{(d-2)d}}$.
\begin{enumerate}
\item\label{3again-1}\textbf{Long unipotent orbits in the cusp:} For every $a\in (-\infty , 0]$, there exists $T=T(d,a, \alpha )$ such that, for every $t\geq T$ and $h \in Hb_{\rho }(-\alpha t) \cap \Fl_{d-1}$, the image \emph{via} the projection $\pi: \Pl_d \to \Ml_d$ of the expanded unipotent orbit 
\begin{equation}\label{eq:uniporb}
\Ol_h(t) = \left\{ SO(d)\cdot \iota(h) \begin{pmatrix} e^{-\frac{\lambda t}{2}}
        \Id_{d-1} & 0 \\ e^{\frac{\mu t}{2}} x^\top & e^{\frac{\mu
            t}{2}} \end{pmatrix}  \Biggm|  x\in \left[ -\frac{1}{2}, \frac{1}{2}\right)^{d-1} \right\}
\end{equation}
is entirely contained in $Hb_{\bar r} (a) \subset \Ml_d$.  
\medskip
\item\label{3again-2}\textbf{Averages over truncated locally symmetric spaces:} Let $f \in C_c(\Ml_d)$ with compact support $K = \supp(f)$ and $\tilde f = f \circ \pi \in C(\Pl_d)$ be its periodic lift. Let $\Fl_0(t)^{-a} = \Phi_t\left( \Fl_0(0)^{-a} \right)$, where $\Fl_0(0)^{-a}$ is defined  as in \eqref{eq:F0t'}.

Let $F_{\tilde f}(t)$ be defined as in \eqref{eq:F_f} and $F_{\tilde f,\alpha}(t)$ be defined by replacing $\Fl_0(t)$ with the truncation $\Fl_0(t)^{\alpha t}$, that is
$$ F_{\tilde f,\alpha}(t) = \oint_{\Fl_0(t)^{\alpha t}} \tilde f d\vol_{\Hl_r(t)} = \oint_{\Ml_{d-1}(-\alpha t)} F(t,\cdot)d\vol_{\Ml_{d-1}}, $$
with $F$ as in \eqref{eq:int-recursion}.
There exists $T_0= T_0(d,\supp(f),\alpha)$ such that  for all $t \ge T_0$
$$
    \left\vert F_{\tilde f,\alpha}(t) -  F_{\tilde f}(t) \right\vert
    \le C(d) \Vert f \Vert_\infty e^{-\vartheta \alpha t}, 
  $$
  where $\vartheta = \frac{1}{2}\sqrt{(d-2)(d-1)}$ and $C(d)$ is a constant depending only on $d$.
\end{enumerate}
\end{thm} 

\begin{proof} \eqref{3again-1} Let $h\in \Fl_{d-1}$ be such that $f_{\rho }(h) \leq -\alpha t$. We prove that the unipotent orbit \eqref{eq:uniporb} is contained in the union of the lifts of $Hb_{\bar r} (a)$, that is, in the $\Gamma$-orbit of $Hb_r (a)$. 

Recall (see paragraph following Notation \ref{notat:alphad}) that the fundamental domain $\Fl_{d-1}$ is at finite Hausdorff distance $c_{d-1}$ from the Weyl chamber  
\begin{equation}
\Wl_{0}=\left\{ \diag (e^{s_1}, \dots, e^{s_{d-1}})\Biggm| s_1\geq s_2\geq \dots \geq s_{d-1}, \sum_{i=1}^{d-1} s_i =0\right\}.
\end{equation}

Thus, every $h\in \Fl_{d-1}$ is within distance at most $c_{d-1}$ from a diagonal matrix in $\Wl_{0}$. In particular, given $h\in \Fl_{d-1}$ with $f_{\rho }(h) \leq -\alpha t$, there exists $\delta = \diag(e^{s_1},\dots,e^{s_{d-1}}) \in \Wl_0$ satisfying $f_{\rho}(\delta) \leq -\alpha t +c_{d-1}$ and $\dist_{\Pl_{d-1}} (h,\delta )\leq c_{d-1}$. The latter inequality implies that the orbit that we obtain by replacing $h$ by $\delta$ in \eqref{eq:uniporb} is within Hausdorff distance $c_{d-1}$ from the orbit of $h$. It suffices therefore to prove that this second orbit is entirely contained in the $\Gamma$-orbit of $Hb_r (a')$ with $a'= a-c_{d-1}$. 
  
According to \eqref{eq:horogamma}, this is equivalent to proving that for every $x\in \left[ -\frac{1}{2}, \frac{1}{2}\right)^{d-1}$ there exists a vector $(w,k)^\top \in \widehat \Z^d$ with $w\in \Z^{d-1} $ and $k\in \Z$ such that the quadratic form 
$$ 
Q_0 \cdot \iota(\delta ) \begin{pmatrix} e^{-\frac{\lambda t}{2}}
        \Id_{d-1} & 0 \\ e^{\frac{\mu t}{2}} x^\top & e^{\frac{\mu
            t}{2}} \end{pmatrix} 
$$
evaluated in $(w,k)$ is at most $e^{\mu a'}$. In what follows, we denote $\frac{e^{\mu a'}}{2}$ by $\Kl$. A sufficient con\-dition for the above to hold is that
$$
\left\{ \begin{array}{ccc}
e^{2s_1} w_1^2 +\cdots + e^{2s_{d-1}} w_{d-1}^2 & \leq & \Kl e^{\lambda t},\\
\left| \langle x, w\rangle +k \right|& \leq & \sqrt{\Kl }e^{-\frac{\mu t}{2}}. \end{array}\right.
$$
If we assume $w_1=0$ the inequalities become 
$$
\left\{ \begin{array}{ccc}
e^{2s_2} w_2^2 +\cdots + e^{2s_{d-1}} w_{d-1}^2 & \leq & \Kl e^{\lambda t},\\
\left| \langle x, w\rangle +k \right|& \leq & \sqrt{\Kl }e^{-\frac{\mu t}{2}}. \end{array}\right.
$$
  To prove that the system of inequalities in $w_2,\dots , w_{d-1}, k$ has an integral solution different from $0$ it suffices to prove that the convex body they define has volume at least $2^{d-1}$, by Minkowski's First Theorem \cite[Theorem 10]{Sie}. Using $s_1= -s_2-\cdots -s_{d-1}$, the volume is $2e^{s_1}\omega_{d-2}\left(\sqrt{\Kl }e^{\frac{\lambda t}{2}} \right)^{d-2}\sqrt{\Kl }e^{-\frac{\mu t}{2}}= 2\omega_{d-2} \Kl^{\frac{d-1}{2}} e^{s_1-\frac{\lambda t}{2}}$, where $\omega_{d-2}$ is the volume of the $(d-2)$-dimensional unit ball. 

Then, for any choice of $\alpha_0 > \frac{\lambda}{2}$, the volume is at least $2^{d-1}$ for large enough $t$, provided $s_1 \ge \alpha_0 t$.
  If $s_1\geq \alpha_0 t$ with
$\alpha_0 >\frac{\lambda}{2}$,
then the volume is at least $2^{d-1}$ for $t$ large enough.

It follows from \eqref{eq:busemrho} that the condition $s_1 \ge \alpha_0 t$ is satisfied if
$$ f_\rho(\delta) = \sqrt{\frac{d-1}{d-2}} \log \det \delta^{\perp e_1} = - \sqrt{\frac{d-1}{d-2}} s_1 \le - \alpha_1 t, $$
with $\alpha_1 = \sqrt{\frac{d-1}{d-2}} \alpha_0$.

For given $\alpha > \frac{1}{2 \sqrt{(d-2)d}}$, we can now choose $\alpha_1 \in \left( \frac{1}{2\sqrt{(d-2)d}},\alpha \right)$, which implies $\alpha_0 = \sqrt{\frac{d-2}{d-1}}\alpha_1 > \frac{\lambda}{2}$. Then, for $t$ large enough, that is, $t \ge T(d,a,\alpha)$ with $T(d, a,\alpha)$ suitably chosen, we have for all $h \in \Fl_{d-1}$ with $f_\rho(h) \le - \alpha t$ that $f_\rho(\delta) \le - \alpha t + c_{d-1} \le - \alpha_1 t$, which implies that $\pi(\Ol_h(t)) \subset Hb_{\bar r} (a)$.
  
\medskip

\eqref{3again-2} For any $a \in (-\infty,0]$ such that $Hb_{\bar r} (a)$ is disjoint from $\supp(f)$, we conclude from \eqref{3again-1} that, for all $t \ge T(d,a,\alpha)$,
$$F_{\tilde f}(t) =\frac{1}{\vol\, \Fl_0(t)} \int_{\Fl_0(t)^{\alpha t}} \tilde f (Q) d\vol_{\Hl_r(t)} (Q). 
$$
Therefore, we have $F_{\tilde f,\alpha}(t) -  F_{\tilde f}(t) = \left( \frac{\vol\, \Fl_0(t)}{\vol\, \Fl_0(t)'} -1 \right)\,  F_{\tilde f}(t) .$
Note that $|F_{\tilde f}(t)| \le \Vert f \Vert_\infty$.
Formula \eqref{eq:volhormod} in Lemma \ref{lem:horovolrel} implies that
$$
\frac{\vol\, \Fl_0(t)}{\vol\, \Fl_0(t)^{\alpha t}} -1= \frac{\vol\, \Fl_0(0)}{\vol\, \Fl_0(0)^{\alpha t}} -1.
$$

We have that $\frac{\vol\, \Fl_0(0) - \vol\, \Fl_0(0)^{\alpha t}}{\vol\, \Fl_0(0)^{\alpha t}}= \frac{\vol\left( \Hlb_{\bar \rho } (-\alpha t )\right)}{\vol\, \Ml_{d-1} - \vol\left( \Hlb_{\bar \rho } (-\alpha t )\right) }\leq 2 \frac{\vol\left( \Hlb_{\bar \rho } (-\alpha t )\right)}{\vol\, \Ml_{d-1} }$ for $t$ large enough, and the latter fraction is, according to Proposition \ref{prop:siegelhorob}, at most $C(d)e^{-\vartheta \alpha t}$ for $t$ large enough, where $C(d)$ is a constant depending only on $d$ and $\vartheta = \frac{1}{2}\sqrt{(d-1)(d-2)}$. 
\end{proof}

\begin{cor}\label{cor:main3again}
For every $\alpha \ge \frac{1}{2} \sqrt{\frac{d}{d-2}}$, in both Theorems \ref{thm:main1again} and \ref{thm:main2again} the average $F_{\tilde f}(t)$ can be replaced by the truncated average $F_{\tilde f,\alpha}(t)$ by modifying the constants accordingly (while keeping the exponent in the convergence rate).  
\end{cor}  
  
\appendix

\section{Relevant results about functions on the real line}
\label{app:realline}

The following result is useful for the proof of Theorem \ref{thm:main1}.

\begin{prop} \label{prop:intexpconseq}
Let $0 < \beta < \alpha$ and $\tilde C > 0$. Let
  $g: \R \to \R$ be a continuous function satisfying
 \begin{equation} \label{eq:intest} 
  \left\vert \int_{-\infty}^T e^{\alpha t} g(t) dt \right\vert \le \tilde C e^{\beta T} \quad \text{for all $T \in \R$.} 
  \end{equation}
Let $\varepsilon > 0$ and $\kappa > 0$, $\varphi = \varphi_{\varepsilon,\kappa,\tilde C}$ be the function
\begin{equation} \label{eq:varphi}
\varphi(t) = \frac{4\tilde C}{ \kappa} e^{-\varepsilon t},
\end{equation} 
and $\Tl = \Tl_{\varepsilon,\kappa} \subset \R$ be the set of parameters $t\in \R$ satisfying
  \begin{equation}\label{eq:ineqgt} 
    \vert g(t) \vert \le \kappa e^{-(\alpha-\beta) t + \varepsilon t}.
  \end{equation}  
  
Then the set $\Tl$ has the property that it intersects any interval $[T,S]$ with $S \ge T+\varphi (T)$ whenever
  \begin{equation}\label{eq:T_0}
    T \ge T_0 = \frac{1}{\varepsilon} \log\left( \frac{2 \tilde C (\beta+\varepsilon)}{\kappa} \right).
 \end{equation}
\end{prop}

\begin{proof}
Let $\varepsilon > 0$ and $\kappa > 0$ be chosen.
Consider a nontrivial interval $[T, T'] \subset \R$ such that for every $t\in [T, T']$, the opposite of inequality \eqref{eq:ineqgt} holds. We have that
$$ \vert g(t) \vert > \kappa e^{-(\alpha-\beta)t+\varepsilon t} \quad
\text{for all $t \in [T,T']$.} $$
In particular, $g$ does not change sign on $[T,T']$, and
we have
\begin{equation*}
\left\vert \int_{T}^{T'} e^{\alpha t} g(t) dt \right\vert =
\int_{T}^{T'} e^{\alpha t} |g(t)| dt > \kappa \int_{T}^{T'} 
e^{(\beta+\varepsilon)t} dt = \frac{\kappa}{\beta+\varepsilon} \left(
e^{(\beta+\varepsilon)T'} - e^{(\beta+\varepsilon)T} \right).
\end{equation*}
Using \eqref{eq:intest}, this implies that
\begin{equation*}
\tilde C e^{\beta T'} \ge \left\vert \int_{T}^{T'} e^{\alpha t} g(t) dt \right\vert - \left\vert \int_{-\infty}^{T} e^{\alpha t} g(t) dt \right\vert
> \frac{\kappa}{\beta+\varepsilon} \left( e^{(\beta+\varepsilon)T'} - e^{(\beta+\varepsilon)T}\right) - \tilde C e^{\beta T},
\end{equation*}
that is
$$
\frac{1}{\kappa'} \left( e^{\beta T'} + e^{\beta T}\right) > e^{(\beta+\varepsilon)T'} - e^{(\beta+\varepsilon)T},
$$
with
\begin{equation} \label{eq:kappa'} 
\kappa' = \frac{\kappa}{\tilde C(\beta +\varepsilon)} > 0. 
\end{equation}
We define $\Delta_T = T' - T > 0$. Using this notation, we obtain
\begin{equation} \label{eq:DeltaTest0}
\frac{1}{\kappa'} e^{-\varepsilon T} \left( e^{(\beta + \varepsilon) \Delta_T} + 1\right)
> \frac{1}{\kappa'} e^{-\varepsilon T} \left( e^{\beta \Delta_T} + 1\right) > e^{(\beta + \varepsilon)\Delta_T} -1.
\end{equation}
For $T \ge T_0 = \frac{1}{\varepsilon}\log\left( \frac{2\tilde C(\beta+\varepsilon)}{\kappa}\right) = \frac{1}{\varepsilon}\log\left(\frac{2}{\kappa'}\right)$ we have
\begin{equation}\label{eq:Tcond}
\kappa' - e^{-\varepsilon T} \ge \frac{\kappa'}{2}.
\end{equation}
Therefore, \eqref{eq:DeltaTest0} implies that
$$ e^{(\beta+\varepsilon)\Delta_T} < 1 + \frac{2 e^{-\varepsilon T}}{\kappa' - e^{- \varepsilon T}}. $$
Using \eqref{eq:Tcond} we conclude that
$$ (\beta + \varepsilon) \Delta_T < 
\log\left( 1 + \frac{2e^{-\varepsilon T}}{\kappa'-e^{-\varepsilon T}} \right) 
\le \frac{2 e^{-\varepsilon T}}{\kappa' - e^{-\varepsilon T}}
\le \frac{4}{\kappa'} e^{- \varepsilon T}. 
$$

Substituting $\kappa'$ from \eqref{eq:kappa'} back, we obtain
\begin{equation}\label{eq:Delta}
\Delta_T < \frac{4 \tilde C}{\kappa} e^{-\varepsilon T} = \varphi_{\varepsilon,\kappa,\tilde C} (T). \end{equation}
\end{proof}

Proposition \ref{prop:intexpconseq} has the following
consequence. This consequence is useful for the proof of Theorem \ref{thm:main2}.

\begin{cor} \label{cor:intexpconseq}
Let $0 < \beta < \alpha$ and $\tilde C > 0$. Let $g: \R \to \R$ be a Lipschitz continuous function with Lipschitz constant
  $L \ge 0$, satisfying
 \begin{equation} \label{eq:intest2} 
  \left\vert \int_{-\infty}^T e^{\alpha t} g(t) dt \right\vert \le \tilde C\, e^{\beta T} \quad \text{for all $T \in \R$.} 
  \end{equation}
  Then we have
  $$ |g(t)| \le (\tilde C + 4 L) e^{-\frac{\alpha-\beta}{2} t} \qquad \text{for all}\,\, t \ge \frac{2}{\alpha-\beta} \log(\alpha+\beta). $$

\end{cor}

\begin{proof}
  We apply Proposition \ref{prop:intexpconseq}.
  Choosing $T_0$ as in \eqref{eq:T_0}, we have that for
  $T \in [T_0,\infty)$ there exists $t \in [T,T+\varphi(T)]$ with
  \begin{equation*}
    |g(T)| \le |g(t)| + (t-T) L \le |g(t)| + L\, \varphi(T)
    \le \kappa e^{-(\alpha-\beta)t + \varepsilon t} + L\, \frac{4 \tilde C}{\kappa}
    e^{-\varepsilon T}.
  \end{equation*}
  The choice $\varepsilon = \frac{\alpha-\beta}{2}$ allows to replace the first term in the sum above by the larger term $\kappa e^{-\frac{\alpha-\beta)}{2}T}$ which ensures that the two exponentials in $T$ decrease with the same speed. Moreover, we choose $\kappa = \tilde C$
  and obtain
  $$ |g(T)| \le \left( \tilde C + 4L \right)
  e^{-\frac{\alpha-\beta}{2}T} $$
  for all
  \begin{equation*}
  T \ge T_0 = \frac{1}{\varepsilon} \log\left(
      \frac{2\tilde C (\beta + \varepsilon)}{\kappa}\right)
  = \frac{2}{\alpha-\beta} \log\left( \alpha+\beta
    \right).
  \end{equation*}
\end{proof}

\section{Primitive lattice point counting and the Riemann zeta function}
\label{app:primlattriem}

In this appendix we provide a proof of the relationship between the asymptotics of the error term of the lattice point counting and that of the error term of primitive lattice point counting, as stated in Proposition \ref{prop:latttoprimlatt}.

\begin{proof}[Proof of Proposition \ref{prop:latttoprimlatt}] 
For $x \ge 0$, let 
$$ r_0(x) =\# \Z^d \cap \partial(\sqrt{x} \Bl^d) \quad \text{and} \quad
r_1(x) = \# \widehat \Z^d \cap \partial(\sqrt{x} \Bl^d). $$
It is easy to see that
\begin{equation} \label{eq:r0} 
r_0(x) = \sum_{k^2 \le x} r_1(x/k^2) = \sum_{n \le x} \tau (n) r_1(x/n), 
\end{equation}
where $\tau : \N \to \R$ is the characteristic function of all square numbers. It is easy to see
that the Dirichlet inverse\footnote{This means that we have $\tau * \nu(n) = \sum_{k | n} \tau (k) \nu(n/k) = \delta_1(n)$.} of $\tau $ is $\nu: \N \to \R$ defined by
$$ \nu(n) = \begin{cases} \mu(k), & \text{if $n$ equals a square $k^2$,} \\ 0, & \text{otherwise,} \end{cases} $$
where $\mu: \N \to \{-1,0,1\}$ is the M\"obius function. Applying M\"obius' generalized inversion formula \cite[p. 40]{Ap}, we conclude that
\begin{equation} \label{eq:r1}
r_1(x) = \sum_{k^2 \le x} \mu(k) r_0(x/k^2). 
\end{equation}
The identities \eqref{eq:r0} and \eqref{eq:r1} allow us to relate the error terms. We have
\begin{multline*}
E_0(\Bl^d,R) + \vol(\Bl^d) R^d = N_0(R) = \sum_{x \le R^2} \sum_{k^2 \le x} r_1(x/k^2) 
= \sum_{k \le R} \sum_{x \le R^2/k^2} r_1(x) = \sum_{k \le R} N_1(R/k) \\
= \sum_{k \le R} \left( E_1(\Bl^d,R/k) + \frac{\vol(\Bl^d)}{\zeta(d)} (R/k)^d \right)
= \left( \sum_{k \le R} E_1(\Bl^d,R/k) \right) +  \frac{\vol(\Bl^d)}{\zeta(d)} R^d 
\sum_{k \le R} \frac{1}{k^d},
\end{multline*}
i.e.,
\begin{equation} \label{eq:E0rep} 
E_0(\Bl^d,R) = \left( \sum_{k \le R} E_1(\Bl^d,R/k) \right) - \frac{\vol(\Bl^d)}{\zeta(d)} 
R^d \sum_{k > R} \frac{1}{k^d}. 
\end{equation}
An analogous calculation leads to
\begin{equation} \label{eq:E1rep}
E_1(\Bl^d,R) = \left( \sum_{k \le R} \mu(k) E_0(\Bl^d,R/k) \right) - \vol(\Bl^d) 
R^d \sum_{k > R} \frac{\mu(k)}{k^d}. 
\end{equation}
Now we estimate the two terms on the right hand side of \eqref{eq:E0rep} and \eqref{eq:E1rep} separately. Assuming \eqref{eq:Ejest}, we obtain for the first term for $R \ge R_j$ and
any $\varepsilon: \N \to [-1,1]$:
\begin{equation*}
\left\vert \sum_{k \le R} \varepsilon(k) E_j(\Bl^d,R/k)  \right\vert|\le \sum_{k \le R}
C_j \left( \frac{R}{k} \right)^\alpha \left( \log \frac{R}{k} \right)^\beta
\le C_j R^\alpha (\log R)^\beta \zeta(\alpha).
\end{equation*}
For the second term, we observe that for any $\varepsilon: \N \to [-1,1]$ and $R > 0$ \footnote{In fact, in the case $\varepsilon=\mu$, we even have $\vol(\Bl^d)R^d\sum_{k > R}\mu(k)/k^d =o(R)$, using partial summation and the fact $\sum_{j \le k} \mu(j) = o(k)$.}
$$ \left\vert \vol(\Bl^d)  R^d \sum_{k > R} \frac{\varepsilon(k)}{k^d} \right\vert
 \le \vol(\Bl^d) R^d \int_R^\infty \frac{dx}{x^d} = \frac{\vol(\Bl^d)}{d-1} R. $$
Using $R_0 \ge 2$, both estimates impy that \eqref{eq:E1-jest} holds with
$$ C_{1-j} = C_j \zeta(\alpha) + \frac{\vol(\Bl^d)}{(d-1)(\log 2)^\beta}. $$

For the proof of \eqref{eq:Ejo-est}, we assume that $E_j(\Bl^d,R) = o(R^\alpha)$ and
write $|E_j(\Bl^d,R)| \le R^\alpha \theta(R)$ with $\theta(R) \to 0$ as $R \to \infty$.
Without loss of generality, we can assume that $\theta$ is positive and monotone decreasing. Since $E_{1-j}(\Bl^d,R) = \left[ \sum_{k \le R} \varepsilon(k) E_j(\Bl^d,R/k) \right] + O(R)$ with $\varepsilon = \mu$ if $j=0$ and $\varepsilon \equiv 1$ if $j=1$, it suffices to show
that
$$ \sum_{k \le R} \varepsilon(k) E_j(\Bl^d,R/k)  = o(R^\alpha). $$
We split the sum into two parts and estimate them separately, i.e., 
$$ \left| \sum_{k \le R} \varepsilon(k) E_j(\Bl^d,R/k) \right| \leq \sum_{k \le \sqrt{R}} \left| E_j(\Bl^d,R/k) \right| + \sum_{\sqrt{R} < k \le R} \left| E_j(\Bl^d,R/k) \right|.  $$
For the first term, using the monotonicity of $\theta$, we obtain
$$ \sum_{k \le \sqrt{R}} \left| E_j(\Bl^d,R/k) \right| \le \sum_{k \le \sqrt{R}} 
\left( \frac{R}{k} \right)^\alpha \theta(R/k) \le \zeta(\alpha) R^\alpha \theta(\sqrt{R})
= o(R^\alpha). $$
For the second term, we use the boundedness of $\theta$ and obtain
$$ \sum_{\sqrt{R} < k \le R} \left| E_j(\Bl^d,R/k) \right| \le \theta(1) R^\alpha \sum_{k > \sqrt{R}} \frac{1}{k^\alpha} = o(R^\alpha). $$
This shows that $E_{1-j}(\Bl^d,R) = o(R^\alpha)$, finishing the proof. 
\end{proof}

\end{document}